\documentclass[11pt, english]{amsart}
\usepackage{amssymb,amsthm,amstext, color}

\usepackage{mathrsfs}
\usepackage{bm}
\usepackage[utf8]{inputenc}
\usepackage{xcolor}

\usepackage{comment} 
\usepackage{amsmath}

\usepackage{youngtab}
\usepackage{colortbl}

\usepackage{geometry}
\usepackage{mathrsfs,upgreek}
\usepackage{appendix}

\input xy
\xyoption{all}

\usepackage{dynkin-diagrams}

\usepackage[only,mapsfrom,heavycircles,lightning]{stmaryrd}
\usepackage{calc}
\usepackage{paralist}

\makeatletter
\ifnum\@ptsize=0  \fi \ifnum\@ptsize=2  \fi 

\makeatletter
\newcommand{\arkfamily}{\fontencoding{U}\fontfamily{ark}\selectfont}
\newcommand{\ark@sym}[1]{{\arkfamily\symbol{#1}}}
\newcommand{\leftthumbsup}{\ark@sym{'125}}
\newcommand{\smallpencil}{\ark@sym{'120}}

\usepackage[]{todonotes}

\newcounter{genus}
\newcounter{gagid}

\@addtoreset{gagid}{section}

\makeatother

\theoremstyle{plain}
\newtheorem{theorem}{Theorem}
\newtheorem*{theorem*}{Theorem}
\newtheorem*{conjecture*}{Conjecture}
\newtheorem{prop}[theorem]{Proposition}
\newtheorem{lemma}[theorem]{Lemma}

\def\CC{{\mathbb{C}}}
\def\GG{{\mathbb{G}}}

\def\OO{{\mathbb{O}}}

\def\PP{{\mathbb{P}}}
\def\QQ{{\mathbb{Q}}}

\def\cO{{\mathcal{O}}}
\def\cE{{\mathcal{E}}}

\def\cI{{\mathcal{I}}}
\def\cL{{\mathcal{L}}}

\def\cS{{\mathcal{S}}}

\def\cU{{\mathcal{U}}}

\def\lra{{\longrightarrow}}
\def\fs{{\mathfrak s}}

\def\fg{{\mathfrak g}}
\def\fso{\mathfrak{so}}
\def\fsp{\mathfrak{sp}}
\def\fsl{\mathfrak{sl}}\def\fgl{\mathfrak{gl}}

\usepackage[hypertexnames=false]{hyperref}

\begin{document}

\title{Spinorial Fano manifolds}
\author{Alessandro Frassineti, Laurent Manivel}

\address{Università di Genova, Via Dodecaneso 35, 16146 Genova, Italy}
\email{alessandro.frassineti@edu.unige.it}

\address{
Institut de Math\'ematiques de Marseille, 
Aix-Marseille University, CNRS, I2M, UMR 7373, Marseille, France}
\email{laurent.manivel@math.cnrs.fr}

\date{}

\begin{abstract}
We construct prime Fano manifolds from spin representations of $Spin_n$ for $n\le 14$. In this range, and if $n\ne 13$, the projectivizations of these representations are prehomogeneous, and we deduce that our Fano manifolds are locally rigid and, up to a few exceptions,  quasi-homogeneous under the action of their automorphism groups. For $n=13$ we obtain a non-trivial family of minimal compactifications of $SL_3\times SL_3$, modulo some finite group.  
\end{abstract}

\maketitle

\section{Introduction and Results}

\subsection*{Prime Fano manifolds} The classification of prime Fano threefolds (whose Picard group has rank one) contains $17$ families \cite{IP}, three of which consist in a single, locally rigid threefold: $\PP^3$, $\QQ^3$ and the quintic del Pezzo threefold. All have big automorphism group, big enough for the action to be homogeneous or at least quasi-homogeneous (there is an open orbit). As Mukai has shown (see e.g. \cite{mukai}), most of the other prime Fano threefolds have natural embeddings into some rational homogeneous spaces of Picard number one, which are also locally rigid prime Fano manifolds (and even globally rigid, up to the exception of the orthogonal Grassmannian $OG(2,7)$).

One can therefore expect that, if Fano manifolds are ever classified in higher dimension, the prime and rigid ones among them will be of special importance. We thus need to enlarge our lists of varieties that can potentially play the same role as homogeneous spaces, typically the quasi-homogeneous ones. In dimension four, restricting to the prime ones, we are only aware of  $\PP^4$, $\QQ^4\simeq G(2,4)$ and the quintic del Pezzo fourfold. In higher dimension, such "completions of homogeneous spaces" have been classified under certain hypothesis \cite{akhiezer, pasquier0, pasquier, ruzzi, ruzzi1}, but only very few of them are prime Fano. 

The main goal of this paper is to construct prime, locally  rigid, quasi-homogeneous (or almost  quasi-homogeneous) Fano manifolds from spin representations.

 \subsection*{Compactifications of homogeneous spaces}
We are interested in Fano varieties obtained as zero loci of sections of homogeneous vector bundles on rational homogeneous spaces $G/P$, where $G$ is a simple complex Lie group and $P$ some parabolic subgroup. Each irreducible $G$-module $V_\lambda$ is, as follows from the Borel-Weil theorem, the space of global sections of an irreducible homogeneous bundle $\cE_\lambda$ on $G/P$. This bundle is automatically globally generated; if its rank is smaller than the dimension of $G/P$, the zero locus $Z(s)$ of a general section $s$ is therefore smooth of the expected dimension, and in most cases, nonempty. Up to isomorphism, $X=Z(s)$ only depends on the $G$-orbit of $[s]$ in $\PP(V_\lambda)$, and there are two typical situations: either the generic orbit has positive codimension, and one gets nontrivial families of varieties; or there is an open orbit, and generically one always gets the same variety $X$.

In the second situation, the $G$-action on $\PP (V_\lambda)$ is prehomogeneous. For $s$ a general section, the stabilizer $H$ of $[s]$ in $G$ acts on $X$ (and is usually equal to its automorphism group). This provides a nice way to construct projective varieties with big automorphism groups. We will discuss a series of cases for which the stabilizer $H$ is so big that it acts on $X$ with an open orbit $H/K$, for which $X$ provides a nice smooth equivariant compactification. 

\subsection*{Local rigidity}\label{locRig}
In the situation where $G$ has an open orbit in $\PP(V_\lambda)$, the variety
$X=Z(s)$ does not depend on the general section $s$ and we can expect it to be locally rigid. This is usually checked by using the normal exact sequence 
$$0\lra TX\lra T(G/P)_{|X}\lra \cE_{\lambda|X}\lra 0.$$
We can suppose that  $H^0(G/P,T(G/P))=\fg$, the Lie algebra of $G$, and $H^q(G/P,T(G/P))=0$ for $q>0$ (see e.g. \cite[section 4.8, Theorem 1 \& Theorem 2]{akhiezer-book}). Suppose we have 
\begin{equation}\label{c1}
 H^0(X,T(G/P)_{|X})\simeq H^0(G/P,T(G/P))=\fg,
\end{equation}
\vspace*{-6mm}
\begin{equation}\label{c11}
 H^1(X,T(G/P)_{|X})\simeq H^1(G/P,T(G/P))=0,
\end{equation}
\vspace*{-6mm}
\begin{equation}\label{c2}
 H^0(X,\cE_{\lambda|X})\simeq H^0(G/P,\cE_{\lambda})/\CC s=V_\lambda/\CC s,
\end{equation}
these isomorphisms being obtained by restriction from $G/P$ to $X$. Then the long exact sequence in cohomology induced by the normal exact sequence reads as follows:
\begin{equation}\label{long}
    0\lra H^0(X,TX)\lra \fg\lra V_\lambda/\CC s\lra H^1(X,TX)\lra 0.
\end{equation}
The middle morphism is nothing else than the differential of the action of $G$ at $[s]$, from $\fg$ to $T_{[s]}\PP(V_\lambda)$. When $[s]$ has an open orbit, this morphism is surjective. Then $H^1(X,TX)=0$, which means that $X$ is locally rigid, as expected. Moreover, $H^0(X,TX)$ is the Lie algebra of the stabilizer $H$ of $[s]$, which means that $H$ is a subgroup of finite index of $Aut(X)$. 

Conditions (\ref{c1}), (\ref{c11}) and (\ref{c2}) are usually verified thanks to the Koszul complex that resolves $\cO_X$, by deducing the required vanishing of cohomology from Bott's theorem. There are some exceptions, so these conditions must be checked carefully. 

\medskip\noindent {\it Remark}. 
In some of the cases we consider, the homogeneous vector bundle $\cE$ is not irreducible but rather
a direct sum of two irreducible vector bundles, and 
condition (\ref{c2}) has to be modified accordingly.
When we have twice the same bundle $\cE_\lambda$, 
with two sections $s_1,s_2$, this condition becomes 
\begin{equation}\label{c2'}
H^0(X,\cE_{\lambda|X})\simeq V_{\lambda}/\langle s_1,s_2\rangle. \tag{$3\,'$}
\end{equation}
Moreover, the long exact sequence (\ref{long}) becomes 
\begin{equation}\label{4'}
    0\lra H^0(X,TX)\lra \fg\lra \langle s_1,s_2\rangle^\vee\otimes  V_{\lambda}/\langle s_1,s_2\rangle\lra H^1(X,TX)\lra 0.\tag{$4\,'$}
\end{equation}
Here, the middle morphism is the differential of the action of $G$ on the Grassmannian of planes in $V_\lambda$, from $\fg$ to $T_{\langle s_1,s_2\rangle}G(2,V_\lambda)$. So we expect $X$ to be locally rigid when $G$ acts on 
$G(2,V_\lambda)$ with an open orbit.

\subsection*{Spin representations and spin bundles}
An important source of quasi-homogeneous representations is provided by spin representations in low dimension. Recall that $Spin_{2n+1}$ has one spin representation $\Delta$ of dimension $2^n$, while $Spin_{2n}$ has two half-spin representations  $\Delta_+$, $\Delta_-$ of dimension $2^{n-1}$. The dimension of $Spin_m$ is bigger than the dimension of its spin representations for $m\le 14$, and the fact is that the action of such a (projectivised) representation is quasi-homogeneous, except for $m=13$. We need a brief reminder on spin representations and spinor bundles. Classical references on spinors are \cite{chevalley,fh}. 

\medskip
Suppose $m=2n$ is even (the discussion would be similar in the odd case). One can define the half-spin representations by fixing a decomposition $V=\CC^{2n}=E\oplus F$ as a sum of two maximal isotropic spaces. The quadratic form on $V$ then provides an identification of $F$ with the dual of $E$; we always use in the sequel a basis $e_1,\ldots ,e_n$ of $E$ with dual basis  $f_1,\ldots ,f_n$ of $F$. The half-spin representations are obtained by letting 
$$\Delta_+=\wedge^{even}E, \qquad \Delta_-=\wedge^{odd}E.$$
Note that there are natural maps $V\otimes \Delta_\pm\lra \Delta_\mp$, where $E$ acts by wedge products and $F\simeq E^\vee$ by contractions. This induces maps  $\wedge^2V\otimes \Delta_\pm\lra \Delta_\pm$ which recover the Lie algebra actions, through the identification of $\wedge^2V$ with $\fso(V)$. 
More generally, for any $k\ge 0$ there are natural maps 
$$\wedge^{2k}V\otimes \Delta_\pm\lra \Delta_\pm, \qquad
\wedge^{2k+1}V\otimes \Delta_\pm\lra \Delta_\mp.$$

For any parabolic subgroup $P$ of $Spin_{2n}$, we can consider $\Delta_\pm$ as the space of global sections of a pair of irreducible bundles $\cE_\pm$ on $Spin_{2n}/P$, that we simply denote by $\cS_\pm$ and call {\it spinor bundles}. We will restrict in the sequel to those $Spin_{2n}/P$ which are orthogonal Grassmannians $OG(k,2n)$, parametrizing isotropic subspaces $U$ of $V$ of dimension $k$ (see \cite{ottaviani} for the special case of spinor bundles on quadrics, e.g. $k=1$). 
For $k<n-1$, the  fiber of $\cS_\pm$ at the point $[U]$ can be described as the cokernel of the morphism 
$$U\otimes \Delta_\pm \hookrightarrow V\otimes\Delta_\pm\lra \Delta_\mp.$$
 We can think of this fiber as a spin representation for the quadratic space $U^\perp / U$. In particular the rank of $\cS_\pm$ is $2^{n-k-1}$, and taking the sum of its weights we easily see that its first Chern class is 
$$c_1(\cS_\pm)=\frac{1}{2}\big(\mathrm{rank}\; \cS_\pm\big) \cL,$$
if $\cL=\cE_{\omega_k}$ denotes the ample generator of the Picard group. 
Moreover, the condition for a section $s$ to vanish at $[U]$
 can be expressed in terms of the Pl\"ucker representative $\omega_U\subset \wedge^kV$ of $U$, as follows. 

\begin{lemma}\label{kills}
Consider a section $s$ of a spinor bundle on $OG(k,m)$, $k<\lfloor \frac{m}{2}\rfloor$,
defined by a spinor $\delta$. Then $s$ vanishes at
$[U]$ when $\delta$ is killed by $\omega_U$.
\end{lemma}

\proof We need to prove that, being given a basis $u_1,\ldots , u_k$ of $U$, our spinor $\delta$ can be decomposed as $u_1.\delta_1+\cdots +u_k.\delta_k$ for some spinors   $\delta_1, \ldots, \delta_k$, if and only if $\omega_U.\delta=0$. Up to the action of $Spin_m$, we can suppose that $U\subset E$ and the basis is $e_1,\ldots , e_k$, for a basis  $e_1,\ldots , e_n$ of $E$. Then the hypothesis is that $\delta\in \wedge^*E$ is a combination of some basis vectors $e_J$ with $J\cap \{1,\ldots ,k\}\ne\emptyset$, and this is clearly equivalent to the condition that the wedge product with $e_1\wedge\cdots\wedge e_k\in \omega_U$ vanishes. \qed 

\medskip
As usual the cases $k=n-1$ and $k=n$ are special. The orthogonal Grassmannian $OG(n,2n)$ has two components $OG(n,2n)_+$ 
and $OG(n,2n)_-$, that can be embedded respectively into $\PP(\Delta_+)$ and $\PP(\Delta_-)$. One of the spinor bundles is the hyperplane bundle for this embedding, while the other one is a twist of the tautological rank $n$ vector bundle, with determinant equal to $n-2$ times the hyperplane bundle. 

On the other hand, the orthogonal Grassmannian $OG(n-1,2n)$ embeds into $OG(n,2n)_+\times OG(n,2n)_-$ (this is the usual property that a $(n-1)$-dimensional isotropic subspace is the intersection of two maximal isotropic subspaces, one from each family); in particular its Picard number is two, and the Picard group is generated by the two spinor bundles, which are the pull-backs of the hyperplane bundles from $OG(n,2n)_+$ and $OG(n,2n)_-$.

\medskip We will denote by $(OG(k,m),\cE)$ the variety defined as the zero locus of a general section of a homogeneous bundle $\cE$, in case $\PP H^0(\cE)$ is prehomogeneous. Most of the time $\cE$ will be a spinor bundle (sometimes a direct sum of spinor bundles) and $H^0(\cE)$ a spin or half-spin representation. The variety  $(OG(k,m),\cE)$ will inherit an action of the generic stabilizer of  $Spin_m$ acting on $\PP H^0(\cE)$, which we need to recall.

\subsection*{Generic isotropy groups} Prehomogeneous irreducible representations of connected linear algebraic groups have been classified in \cite{sk}. This includes the actions of $GL_1\times   Spin_m$ on a spin representation, which is prehomogeneous exactly when the action of $Spin_m$ on the projectivized spin representation is prehomogeneous - this is the condition we are interested in. Generic stabilizers have been computed explicitly; their Lie algebras can be found in  \cite{sk}, and the actual groups in \cite[Table 1]{gg}. The result of their computations is the following:

$$\begin{array}{ccc}
m &\qquad & \mathrm{stabilizer} \\
7 && G_2 \\
8& &Spin_7 \\
9& &Spin_7 \\
10 && (\CC^*\times Spin_7)\rtimes\Delta \\
11 && SL_5 \\
12&& SL_6 \\ 
14 && G_2\times G_2
\end{array}$$
\begin{center} \small Table 1. Generic stabilizers \normalsize\end{center}

\medskip
The case where $m=13$ is also interesting, although in this dimension the projectivized spin representation is not prehomogeneous. In fact, the generic stabilizer is $SL_3\times SL_3$, and the generic orbit has codimension one \cite{gg}. 

\medskip
Note that the stabilizer is always reductive, which implies that the complement of the open orbit is a hypersurface - except for the well-known case where $m=10$, for which the complement of the open orbit is just the closed orbit $OG(5,10)_\pm$.

\subsection*{Results I: Unexpected identifications} In some cases, the varieties $(OG(k,m),\cE)$ are homogeneous, leading to some unexpected identifications. We get the following ones:
$$(OG(2,7),\cS)\simeq G_2/P_2,$$
$$(OG(2,9),\cS)\simeq (OG(2,8),\cS_\pm)\simeq OG(2,7),$$
$$(OG(2,10),\cS_+\oplus\cS_+)\simeq G_2/P_2, \qquad 
(OG(2,10),\cS_+\oplus\cS_-)\simeq Fl_4,$$
$$(OG(2,11),\cS)\simeq Fl_5, \qquad (OG(2,12),\cS_\pm)\simeq Fl_6,$$
$$(OG(2,13),\cS)\simeq Fl_3\cup Fl_3, \qquad (OG(2,14),\cS_\pm)\simeq G_2/P_2\cup G_2/P_2.$$

\medskip
Note that these identifications all involve {\it adjoint varieties}, the adjoint variety of a simple Lie group $G$ being the unique close orbit inside $\PP(\fg)$, its projectivized Lie algebra. Indeed the adjoint variety of $SL_n$ is $Fl_n$ (the incidence quadric in $\PP^{n-1}\times\check{\PP}^{n-1}$), that of $Spin_m$ is $OG(2,m)$, and the adjoint variety of $G_2$ is $G_2/P_2$, $P_2$ being the parabolic subgroup associated to the long simple root.  

The identification $(OG(2,9),\cS)\simeq OG(2,7)$ is especially interesting. There is an open orbit of sections $s$ of $\cS$ whose zero-locus $Z(s)$ is isomorphic to $OG(2,7)$, but there is also a hypersurface of sections for which $Z(s)$ is still smooth, but not isomorphic to the orthogonal Grassmannian. We thus recover the suprising property of $OG(2,7)$, first observed in \cite{pasper} and further analyzed in \cite{hwang-li}, to degenerate to a smooth variety (the $G_2$-horospherical variety). This phenomenon is unique among rational homogeneous spaces of Picard rank one. 

\smallskip
For the spinor varieties $OG(n,2n)_\pm$ we get the two identifications
$$(OG(5,10)_-,\cS_+)=\QQ^5, \qquad (OG(6,12)_-,\cS_+)=G(3,6).$$

We also include a curious identification  $$(OG(n,2n+1),\wedge^2\cU^\vee)\simeq (\PP^1)^n$$
which is the analogue of a classical one for Lagrangian Grassmannians (parametrizing maximal isotropic spaces with respect to some symplectic form), namely (see \cite[Theorem 3.1]{kuznetsov})
$$(LG(n,2n),\wedge^2\cU^\vee)\simeq (\PP^1)^n.$$
The latter statement can easily be deduced from the fact that two general symplectic forms can be simultaneously "diagonalized" by planes. 

\subsection*{Results II: Spinorial Fano manifolds}
In the remaining cases where $(OG(k,m),\cE)$ is not homogeneous, 
we obtain interesting Fano manifolds, that we expect to be locally rigid, and to have sufficiently big automorphism group to be quasi-homogeneous, or close. Those are our {\it spinorial Fano manifolds}, whose list is given in Table 2;  for each such variety $X$, we provide its dimension $d_X$, its Picard number $\rho_X$ and its index $\iota_X$.

\medskip
Some of these spinorial Fano manifolds have already appeared in the literature. 
The hyperplane section $ (OG(5,10)_+,\cS_+)$ appears in \cite{pasquier} as a horospherical variety (case 2 of Theorem 0.1); its local rigidity was proved in \cite{pasper}. The codimension two linear section $ (OG(5,10)_+,\cS_+\oplus \cS_+)$ is discussed in 
\cite[Proposition 4.8]{bfm}, where it is shown to admit a quasi-homogeneous action of $G_2\times SL_2$. 
The hyperplane section $ (OG(6,12)_+,\cS_+)$ appears in \cite{bucz} as a quasi-homogeneous Legendrian variety. 
The hyperplane section $ (OG(7,14)_+,\cS_+)$ is discussed in \cite[Proposition 5.3]{bfm}, where it is proved to be quasi-homogeneous under an action of $G_2\times G_2$.
The local rigidity for all these cases follows from \cite[Theorem 1.1]{bfm}.

Finally, the last case of the list, $(OG(7,14)_-, \cS_+)$, is the {\it double Cayley Grassmannian} already studied in 
\cite{doublecayley}, where is is shown to be an equivariant compactification of $G_2$.
\medskip

\begin{center}
\begin{tabular}{ccccl}
$X$ & $d_X$ &  $\rho_X$ &  $\iota_X$ &    \\
$(OG(3,8), \cS_+)$ & $8$& $2$ & $1$ & \\
 $(OG(3,8), \cS_+\oplus \cS_-)$    & $7$&$2$&$3$ &\\
 $(OG(3,9), \cS)$ & $10$ & $1$ & $4$  &\\
  $(OG(2,10), \cS_+)$ & $9$ & $1$ & $5$ &\\
    $(OG(3,10), \cS_+)$   & $13$&$1$&$5$& \\
       $(OG(3,10), \cS_+\oplus \cS_+)$   & $11$&$1$&$4$& \\
   $(OG(3,10), \cS_+\oplus \cS_-)$   & $11$&$1$&$4$& \\
   $(OG(5,10)_+,\cS_+)$ & $9$&$1$ & $7$ &\cite{pasquier}\\
    $(OG(5,10)_+,\cS_+\oplus\cS_+)$ & $8$&$1$ & $6$&\cite{bfm}\\
   $(OG(3,11), \cS)$ & $14$ & $1$ & $5$ & \\ 
   $(OG(3,12), \cS_+)$ & $17$ & $1$ & $6$ &  \\
    $(OG(4,12), \cS_+)$ & $20$ & $1$ & $6$ &  \\
    $(OG(6,12)_+,\cS_+)$ & $14$ & $1$ & $9$ &\cite{bucz}\\
     $(OG(3,14), \cS_+)$ & $19$ & $1$ & $6$  & \\
    $(OG(4,14), \cS_+)$ & $26$ & $1$ & $7$   &\\  
    $(OG(5,14), \cS_+)$ & $28$ & $1$ & $7$   &\\  
    $(OG(7,14)_+, \cS_+)$ & $20$ & $1$ & $11$&\cite{bfm}\\
    $(OG(7,14)_-, \cS_+)$  & $14$ & $1$ & $7$ &\cite{doublecayley}
\end{tabular}
\end{center}
\begin{center} \small Table 2. Spinorial Fano manifolds  \normalsize\end{center}

\medskip\noindent {\bf Theorem A.} {\it Each of these spinorial Fano manifolds is locally rigid.}

\medskip\noindent {\bf Theorem B.} {\it Each of these spinorial Fano manifolds is quasi-homogeneous, except possibly for $(OG(k,14), \cS_+)$ for $k=4,5$, whose cohomogeneity is at most one.}

\bigskip
The verification of the quasi-homogeneity property must be done case by case. We will work at the Lie algebra level and use extensively the explicit computations made in \cite{sk}, where the stabilizer of a generic (explicit) spinor (in each relevant dimension) is described. This spinor defines a copy of our Fano manifold, in which we will find a generic point and compute its stabilizer. Finding a generic point can sometimes be challenging; computing its stabilizer, hence checking that its orbit has the correct dimension, is then just linear algebra. In some cases we find that the generic stabilizer is reductive: for example, up to some finite group that we did not identify, we show that $(OG(3,12),\cS_+)$ is a smooth compactifications of  $G_2\times G_2/SL_2\times SL_2\times SL_2$. 

\medskip
On the other hand, the verification of the rigidity property always follows the same strategy: we check conditions (\ref{c1}), (\ref{c11}) and  (\ref{c2}) through the Koszul complex defined by our vector bundle section; this reduces us to a bunch of cohomological conditions that can be verified by applying Bott's theorem. We will explain only a few cases in detail: Proposition \ref{rigidity} exhibits the standard situation, where the expected cohomological vanishing hold true; Propositions \ref{rigidityOG311} and \ref{rigidityOG312} discuss the two cases for which some unexpected non-vanishing has to be dealt with. The other cases have been treated with the help of a computer program written in \textit{Macaulay2}. 

\medskip
Similarly, we compute the Betti numbers of our manifolds by following a uniform pattern, as follows. For a given variety $X = (Y,\mathcal{E})$, we start from the conormal sequence 
\[
0\longrightarrow \mathcal{E}_{|X}^\vee \longrightarrow \Omega_{Y|X} \longrightarrow \Omega_X\longrightarrow 0
\]
and its $p$-th exterior power
\begin{align*}
0&\longrightarrow \mathrm{Sym}^p\mathcal{E}_{|X}^\vee \longrightarrow \mathrm{Sym}^{p-1}\mathcal{E}_{|X}^\vee\otimes\Omega_{Y|X} \longrightarrow \cdots  \hspace*{3cm} \\ 
\hspace*{3cm}
\cdots &\longrightarrow  \mathcal{E}_{|X}^\vee \otimes \Omega_{Y|X}^{p-1}  \longrightarrow \Omega_{Y|X}^{p} \longrightarrow        \Omega_X^{p} \longrightarrow 0
\end{align*}
In order to compute the Hodge numbers $h^{p,q}(X)$, we use the Koszul complex defined by our vector bundle section, and Bott's theorem to compute the cohomology of $\mathrm{Sym}^j\mathcal{E}_{|X}^\vee \otimes \Omega_{Y|X}^{p-j}$. In most cases, enough terms are vanishing for the computation to be successful, and the cohomology turns out to be pure 
(in the sense that $h^{p,q}(X)=0$ for $p\ne q$). We thus get the Betti numbers as $b_{2p+1}(X)=0$ and $b_{2p}(X)=h^{p,p}(X)$; for each of our spinorial Fano manifolds we provide the explicit list of the even Betti numbers we are able to compute (in some cases the codimension is to big for the computation to be effective).

\bigskip 
In the case of $Spin_{13}$, whose action on the projectivized spin representation has cohomogeneity one, we can deduce from the same constructions as before, several families of Fano manifolds, that we briefly discuss  in the final section of this paper. The most striking case is that of $(OG(3,13),\cS)$, which yields a nontrivial family of prime quasi-homogeneous Fano manifolds of dimension $16$: again modulo some finite group, they are all minimal compactifications of $SL_3\times SL_3$.

\medskip\noindent 
 {\bf Acknowledgements.} We thank Sasha Kuznetsov for useful discussions, that led to Proposition \ref{horo}. A. Frassineti was partially supported by INdAM - GNSAGA Project, Varietà di Fano e hyperkahler: costruzioni, classificazioni e collegamenti (CUP E53C24001950001). L. Manivel
was supported by the project FanoHK ANR-20-CE40-0023.

\section{Proofs} 

\begin{prop}\label{G2P2}
$(OG(2,7), \cS)=G_2/P_2$
\end{prop}

\proof By adjunction, $(OG(2,7), \cS)$ is a Fano fivefold $X$ of index $3$. A section of $\cS$ is a vector in the spin representation $\Delta$ of $Spin_7$, whose generic stabilizer is of type $G_2$; this implies that $X$ admits an action of $G_2$. But the 
minimal dimension of a projective variety with such an action is five, since this is the dimension of the two $G_2$-Grassmannians $G_2/P_1\simeq\QQ^5$ and $G_2/P_2$. 
Since the quadric has index $5$ while $X$ has index three, $X$ has to coincide with 
$G_2/P_2$. \qed

\begin{prop}
$(OG(2,8), \cS_+)=OG(2,7)$
\end{prop}

\proof The outer automorphisms of $Spin_8$ induced by triality (see e.g \cite[§20.3]{fh}) permute its three irreducible eight-dimensional representations, namely the natural representation $V_8$ and the two half-spin representations $\Delta_+$ and $\Delta_-$. 
Since $OG(2,8)$ is the closed orbit inside $\PP(\fso_8)$, the projectivization of the Lie algebra, it is preserved by these outer automorphisms, which exchange correspondingly its three irreducible rank two vector bundles, namely the (dual) tautological vector bundle
$\cU^\vee$, and the two spinor bundles $\cS_+$ and $\cS_-$ (whose spaces of global sections are precisely $V_8, \Delta_+$ and $\Delta_-$):
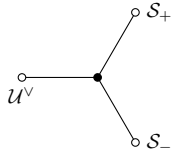
$$\dynkin[labels={\cU^\vee, , \cS_+, \cS_-}, edge length = 1cm] D{o*oo}$$

\medskip
We can therefore use triality, which exchanges the outer vertices of the Dynkin diagram, in order to identify  $(OG(2,8), \cS_+)$ with $(OG(2,8),\cU^\vee)$. And the latter is obviously $OG(2,7)$. \qed 

\medskip Now we turn to $OG(3,8)$, which has Picard rank two. As we already mentioned, this can be interpreted as follows: each isotropic $3$-plane $P$ is contained in two maximal isotropic spaces $Q_+$ and $Q_-$ from the two families $OG(4,8)_+$ and $OG(4,8)_-$. The latter are both  $6$-dimensional quadrics (again by triality), and we denote them by $\QQ_+^6$ and $\QQ_-^6$. Then $OG(3,8)$ embed inside $\QQ_+^6\times\QQ_-^6$, the fibers of both projections being projective planes.  Moreover, the two spinor bundles $\cS_+$ and $\cS_-$ are nothing else than the line bundles obtained by pulling-back the minimal ample line bundles from the two quadrics.

\medskip\noindent {\it Remark.}
Since $OG(3,8)$ has codimension $3$ and is subcanonical in $\QQ_+^6\times\QQ_-^6$, it has to be a Pfaffian degeneracy locus. Indeed, if $\cU_+$ and $\cU_-$ denote the rank four tautological bundles on $OG(4,8)_+$ and $OG(4,8)_-$, we can describe $OG(3,8)$ as a degeneracy locus for the morphism $\cU_+\lra V_8/\cU_-\simeq \cU_-^\vee$, or equivalently $\cU_-\lra V_8/\cU_+\simeq \cU_-^\vee$, or equivalently  $\cU_+\oplus \cU_-\lra (\cU_+\oplus \cU_-)^\vee$. The latter, suitably normalized, becomes skew-symmetric and is the only one to have a degeneracy locus of the expected codimension. 

\begin{prop}
$(OG(3,8), \cS_+)$ is a Fano eightfold with a quasi-homogeneous action of $Spin_7$. 
\end{prop}

\noindent {\it Betti numbers:} $(1,2,3,4,4,4,3,2,1)$.

\proof The stabilizer in $\fso_8$ of the generic spinor $\delta=1+e_{1234}$ in $\Delta_+$ is computed in \cite{sk} (see the proof of Proposition 35). Consider in $\CC^8$ the three-dimensional isotropic space $P$ generated by $e_1+f_2, e_2-f_1, e_3+f_4$. A straightforward computation shows that $\omega_P.\delta=0$, which means by Lemma \ref{kills} that $[P]$ belongs to the copy of  $(OG(3,8), \cS_+)$ defined by $\delta$. Using the formulas given in \cite{sk}, one can compute the dimension of the stabilizer $\Gamma_P$ of $P$ in $Spin_7$, or rather its Lie algebra $\gamma_P$. One checks explicitly that there is an exact sequence 
$$0\lra V\lra\gamma_P\lra \fgl(P)\lra 0,$$
where the induced action of $\fgl(P)$ on $V$ yields an identification $V\simeq P\oplus \CC$. In particular $\gamma_P\simeq\fgl(P)\rtimes V$ has dimension $9+4=13$. Therefore, the $Spin_7$-orbit of $P$ has dimension $21-13=8$ and must be dense in $(OG(3,8), \cS_+)$. 

\medskip In order to prove local rigidity, we verify conditions (\ref{c1}), (\ref{c11}) and (\ref{c2}). The third one follows immediately from the Koszul complex  
$$0\lra \cS_+^\vee\lra \cO\lra \cO_X\lra 0,$$
twisted by $\cS_+$. In order to check the other ones, we rather twist it by $T$, the tangent bundle to $OG(3,8)$, which is not an irreducible homogeneous bundle but an extension 
$$0\lra Hom(\cU,\cU^\perp/\cU)\lra T\lra \wedge^2\cU^\vee\lra 0,$$
where $\cU^\perp/\cU\simeq Hom(\cS_+,\cS_-)\oplus Hom(\cS_-,\cS_+)$. 
In terms of homogeneous bundles, denoting by $\cE_\lambda$ the irreducible homogeneous bundle of highest weight $\lambda$, we have the following correspondence:
$$\dynkin[labels={\cU^\vee\simeq \cE_{\omega_1}, , \cS_+\simeq \cE_{\omega_4}, \cS_-\simeq \cE_{\omega_3}}, edge length = 1cm] D{oo**}$$

\medskip
According to Bott's theorem, $\wedge^2\cU^\vee\simeq \cE_{\omega_2}$ has no higher cohomology, while its space of global sections is $V_{\omega_2}\simeq \fso_8$. Moreover its twist by the line bundle $\cS_+^\vee\simeq \cE_{-\omega_4}$ is acyclic. On the other hand, by Bott's theorem again  $Hom(\cU,\cU^\perp/\cU)\simeq \cE_{\omega_1-\omega_3+\omega_4}\oplus \cE_{\omega_1-\omega_3+\omega_4}$ is acyclic, as well as its twist by $\cS_+^\vee$. This implies (\ref{c1}) and (\ref{c11}), hence the local rigidity of $X$.  
\qed

\medskip\noindent {\it Remark.} Since $\Gamma_P$ is not reductive, the open orbit $Spin_7/\Gamma_P$ is not affine. It would be interesting to describe its complement, and more generally the orbit-structure of the action of the automorphism group, which presumably coincides with $\Gamma_P$.  The same question can be raised for each of our spinorial Fano manifolds.

\medskip\noindent {\it Remark.}
The incidence variety of pairs $(L\subset P)$ inside $OG(2,8)\times (OG(3,8), \cS_+)$ 
is $\PP^2$-bundle over $(OG(3,8), \cS_+)$. The projection to $OG(2,8)$ has fibers isomorphic to $\PP^1\times\PP^1$ over $(OG(2,8),\cS_+)=OG(2,7)$, and fibers $\PP^1$
outside this locus. This gives another way to compute the Poincar\'e polynomial of  $(OG(3,8), \cS_+)$. 

\begin{prop}
$(OG(3,8), \cS_+\oplus \cS_-)$ is a Fano sevenfold of index $3$ and Picard number $2$, with a quasi-homogeneous action of  $G_2$. 
\end{prop}

\noindent {\it Betti numbers}: $(1,2,3,3,3,3,2,1)$. 

\proof Given two generic spinors $\delta_+\in\Delta_+$ and $\delta_-\in\Delta_-$, their common stabilizer in $Spin_8$ is isomorphic to $G_2$. This easily follows
from the description of triality in terms of octonions (one can suppose that $\delta_+$ and $\delta_-$ are equal to $1\in\OO$), or from the explicit computations, at the Lie algebra level, made in \cite{sk} (proof of Proposition 35). In fact these authors compute the common stabilizer of the spinor $\delta_+=1+e_{1234}$ and the vector $v=e_4-f_4$; but this is also the common stabilizer of $\delta_+$ and $\delta_-=v.\delta_+=e_4+e_{123}$. 

We claim that if $P=\langle e_1,f_2,e_4+f_3\rangle$, 
both $\delta_+$ and $\delta_-$ are killed by $\omega_P$. In other words, $[P]$ belongs to $X$. Moreover, using the explicit matrix \cite[(5.32)]{sk}, it is easy to check that the stabilizer of $P$ in $\fg_2$ has codimension $7$, which coincides with the dimension of $X$; so the orbit of $[P]$ must be dense in $X$. \qed

\medskip\noindent {\it Remark}.
We observed in the proof that the action of $Spin_8$ on pairs of spinors (up to scalar) is quasi-transitive: the homogeneous space 
$$Spin_8/G_2\hookrightarrow \PP(\Delta_+)\times \PP(\Delta_-)$$
embeds as an open and dense orbit.

\medskip\noindent {\it Remark}.
The $G_2$-action on $(OG(3,8), \cS_+\oplus \cS_-)$ being quasi-homogeneous, the open orbit must be $G_2/H$ for $H$ a closed subgroup of a  parabolic subgroup $P$ of $G_2$: indeed, by Dynkin's classification, 
the maximal connected subgroups of $G_2$ are the parabolics and subgroups of types  $A_2, A_1\times A_1, A_1$; only $A_2$ has dimension bigger than six, but it has no codimension one closed subgroup, hence the claim. 
As a consequence, we should have a natural rational map $X\dashrightarrow G_2/P$. But this is clear, since we have an embedding $X\hookrightarrow \QQ^5\times\QQ^5$, compatible with the action of $G_2$ on $\QQ^5\simeq G_2/P_1$.  

In fact, the explicit computation shows that $H\simeq \mathbb{G}_m^2\rtimes U_5$, with $U_5$ unipotent.

\medskip\noindent {\it Remark.} Taking three sections of $\cS_+$ and $\cS_-$, we get a Fano threefold of Picard number two, with two birational maps to $\QQ^3$. The locus in $\QQ^3$ where the projection has non trivial fibers (all projective lines, in general) is a degeneracy locus of a morphism 
$$\wedge^3\cU\lra A_3\otimes \cO(-H),$$
where $H$ is the hyperplane class for the spinorial embedding, 
restricted to $\QQ^3\subset\QQ^6=OG(4,8)_+$.
By the Thom-Porteous formula, the class of this locus is the second Segre class of $\wedge^3\cU^\vee(-H)=\cU(H)$ 
(beware that $\det(\cU)=\cO(-2H)$). A computation shows that this can be expressed in terms of Schubert classes from $G(4,8)$ as $s_2(\cU(H))=\sigma_2$. Therefore, the exceptional locus of our projection is a smooth curve $C$ in $\QQ^3$ of degree 
$$d=\int_{OG(4,8)_+}s_2(\cU(H))H^4=\frac{1}{2}\int_{G(4,8)}\sigma_{4321}\sigma_2)\sigma_1^4=12$$
(recall that $H=\sigma_1/2$, and  $c_{top}(S^2\cU^\vee)=16\sigma_{4321}$ is the class of the disjoint union of $OG(4,8)_+$ and $OG(4,8)_-$, which must be corrected by dividing by two). Finally, $C$ is a degree four rational curve in $\PP^4$, and we recover case 
$(21)$ of  the Mori-Mukai classification \cite{mori-mukai}. Moreover, from the two projections we get a nice birational endomorphism of $\QQ^3$, completely similar to the cubo-cubic very special Cremona transformation of \cite{katz}.

\begin{prop}\label{OG27}
$(OG(2,9), \cS)=OG(2,7)$
\end{prop}

\proof A section of the rank four vector bundle $\cS$ is defined by an element of the spin representation $\Delta$ of $Spin_9$, whose dimension if $16$. According to 
\cite[Proposition 35]{sk}, the generic stabilizer $H$ of the action of $\CC^*\times Spin_9$ on $\Delta$ is locally isomorphic to $Spin_7$. More precisely, there is a subgroup $Spin_8\subset Spin_9$, stabilizing some hyperplane $V_8$ in $V_9$, such that $H\simeq Spin_7$ embeds in this $Spin_8$ in a nonstandard way. Hence an action of $Spin_7$ on $(OG(2,9), \cS)$.

Since $OG(2,9)$ is the adjoint variety of $Spin_9$, otherwise said, the unique closed orbit in $\PP(\fso_9)$, we should restrict this action to $Spin_7$. If we first restrict it to $Spin_8$,  the adjoint representation $\fso_9\simeq \wedge^2V_9$ decomposes into $\fso_8\oplus V_8$. 
Then $Spin_7$ embeds in $Spin_8$ in a nonstandard way, which means that $V_8$ is a spin representation of $Spin_7$; but the adjoint representation of $Spin_8$ certainly contains that of $Spin_7$, and must decompose as $\fso_7\oplus V_7$. We claim that the closed orbit $OG(2,7)$ of $Spin_7$ in $\PP(\fso_7)$ is contained in 
$(OG(2,9), \cS)$; this is enough to imply that they are in fact equal. 

In order to prove the claim,  we need a closer look to the proof of \cite[Proposition 35]{sk}. They choose an explicit element in $\Delta$, namely $\delta=1+e_{1234}$, 
and compute explicitly its stabilizer as the space of matrices of the form
$$M= \begin{pmatrix}
A & B \\ B^\dagger & -A^t
\end{pmatrix}$$
with $A\in\fsl_4$, $B$ skew-symmetric, and  $B^\dagger$ the symmetric of $B$ with respect to the antidiagonal. This recovers $\fso_7$ in the nonstandard  model 
$$\fso_7=\fsl_4\oplus \wedge^2V_4.$$
Here the maximal rank subalgebra $\fsl_4$ is generated by the long root spaces of $\fso_7$, while   $\wedge^2V_4$ is spanned by the six short root spaces. In particular, in this model the highest root  of $\fso_7$ coincides with the highest root of $\fsl_4$, and the highest root space may therefore be generated by $e_4^\vee\otimes e_1$. The corresponding element of $\fso_9$ is $e_1\wedge f_4$, which can of course be identified with the isotropic plane $\langle e_1,f_4\rangle$. This plane belongs to $(OG(2,9), \cS)$, more precisely, to the zero locus of the section defined by $\delta$,  because 
$e_1\wedge f_4$ obviously kills $\delta$. But then its whole $Spin_7$-orbit must be 
contained in 
 $(OG(2,9), \cS)$, and this orbit is $OG(2,7)$.
\qed 

\medskip Using e.g. \cite{Lie}, we get the decomposition 
\begin{equation}\label{tenquadrics} S^2\Delta^\vee=V_{2\omega_4}\oplus V_{\omega_1}\oplus \CC .
\end{equation}
This shows that the spinor variety $OG(4,9)$ is cut-out in $\PP(\Delta)$ by ten quadrics: nine quadrics parametrized by the natural representation $V_{\omega_1}\simeq\CC^9$, and an invariant quadric $Q$.  According to \cite[Proposition 5]{igusa}, the orbit decomposition in $\PP(\Delta)$ is extremely simple:

\begin{lemma}\label{orbitsPDelta} The action of $Spin_9$ on $\PP(\Delta)$ has only three orbits, namely: 
\begin{enumerate}
    \item the spinor variety $OG(4,9)$,
    \item the complement of $OG(4,9)$ in $Q$,
    \item the complement of $Q$.
\end{enumerate}
\end{lemma} 

The naive expectation would be that if a spinor $\delta$ does not belong to the open orbit, the zero-locus in $OG(2,9)$ of the corresponding section of $\cS$ should be singular. Surprisingly, this is not the case. In fact, $OG(2,7)$ exhibits this unique property among rational homogeneous spaces of Picard number one, not to be globally rigid: as already noticed in \cite[Proposition 2.3]{pasper}, it admits a smooth degeneration to the so-called $G_2$-horospherical variety. We will prove the following:

\begin{prop}\label{horo} For $[\delta]\in Q$, not a pure spinor, the zero-locus $X_\delta$ in $OG(2,9)$ of the corresponding section of $\cS$ is a copy of the $G_2$-horospherical variety. \end{prop}

This provides a simple and natural explanation for the existence of this degeneration, and answers a question by A. Kuznetsov. In fact, we will only need to prove that $X_\delta$ is smooth: indeed, by a result of \cite{hwang-li}, the only possible smooth degenerations of $OG(2,7)$ are $OG(2,7)$ itself, and the $G_2$-horospherical variety. They can be distinguished by their automorphism groups, which for the latter variety is, according to \cite[Theorem 1.11]{pasquier}, a semidirect product $(G_2\times \GG_m)\rtimes V_7$, where $V_7$ denotes the seven-dimensional irreducible representation of $G_2$. According to \cite[Proposition]{igusa}, this is also the stabilizer in $Spin_9$ of a generic $[\delta]\in Q$, which as usual will be identified with the automorphism group of $X_\delta.$ So we just need to prove the following:

\begin{lemma}
Suppose that $\delta\in\Delta$ is a nonzero spinor whose associated $X_\delta\subset\PP(\Delta)$ is singular. Then $\delta$ is pure.
\end{lemma}

\proof
Recall that we identify $\Delta$ with $\wedge^\bullet E$, for $E$ a maximal isotropic subspace of $V_9$, with basis $e_1,\ldots , e_4$. A spinor $\delta = \sum_I\delta_Ie_I$ is pure when the kernel of the natural map
\begin{align*}
    m_\delta : V_9 &\longrightarrow \Delta\\
    e + f + av_9 &\longmapsto e\wedge \delta + f\lrcorner\, \delta + \sum_I (-1)^{|I|}a\delta_Ie_I
\end{align*}
is a maximal isotropic subspace. This subspace is uniquely determined by $\delta$. The projectivization of the locus of pure spinors is thus identified with  the orthogonal Grassmannian in its spinorial embedding $OG(4,9) \hookrightarrow \mathbb{P}(\Delta)$.

Suppose that $X_\delta$ is singular; by homogeneity of $OG(2,9)$ we may suppose it is singular at point defined by the isotropic plane $P=\langle e_1,e_2\rangle$. In particular, that  $X_\delta$ contains this point means that $\delta$ has zero coefficient $\delta_I$ on $e_I$ when $I$ does not contain $1$ or $2$. That $X_\delta$ is singular at $[P]$ is detected by the non-surjectivity of the natural map $T_{[P]}OG(2,9)\lra \cS_{[P]}$, which can be defined as follows: if a tangent vector is represented by $X\in\fso_9$, its image is simply the class of the spinor $X.\delta$ modulo $P$. This is easy to compute explicitly, and we conclude that the smoothness of $X_\delta$ at  $[P]$ is equivalent to the condition that the following matrix has full rank:

$$\begin{pmatrix} 
\delta_{12} & \delta_{1} & \delta_{2} & 0& 0& 0 
   & 0 & \delta_{13} & \delta_{14} & \delta_{23} & \delta_{24} \\
\delta_{123} & -\delta_{13} & -\delta_{23} &\delta_1&\delta_2&0 
 &0  &0 & -\delta_{134} & 0 & -\delta_{234} \\
\delta_{124} & -\delta_{14} & -\delta_{24} &0&0&-\delta_1  
 &-\delta_2  & \delta_{134} & 0 & \delta_{234} & 0 \\
\delta_{1234} & \delta_{134} & \delta_{234} & \delta_{14} & \delta_{24} & 
 \delta_{13}  & \delta_{23}  &0&0&0&0
\end{pmatrix}$$

\medskip 

We can compute with the help of \textit{Macaulay2} the radical of the ideal generated by the $4$-minors of this matrix in $\mathbb{P}(\Delta)$. It turns out that this ideal contains the ideal cut out by the ten quadrics that define  $OG(4,9)$. They are respectively 
\begin{align*}
    J = \langle &\delta_0, \delta_3, \delta_4,\delta_{34},  
    \delta_{24}\delta_{134} - \delta_{14}\delta_{234}, \delta_{23}\delta_{134} - \delta_{13}\delta_{234}, \delta_2\delta_{134} - \delta_1\delta_{234}, 
    \\&\delta_{24}\delta_{123} - \delta_{23}\delta_{124} + \delta_{12}\delta_{234} - \delta_2\delta_{1234},\delta_{14}\delta_{123} - \delta_{13}\delta_{124} + \delta_{12}\delta_{134} - \delta_1\delta_{1234},  \\
    &
    \delta_{14}\delta_{23} - \delta_{13}\delta_{24}, \delta_2\delta_{14} - \delta_1\delta_{24}, \delta_2\delta_{13} - \delta_1\delta_{23} \rangle
\end{align*}
and
\begin{align*}
    I = \langle &\delta_{34}\delta_{124} - \delta_{24}\delta_{134} + \delta_{14}\delta_{234} - \delta_4\delta_{1234}, 
    \delta_{34}\delta_{123} - \delta_{23}\delta_{134} + \delta_{13}\delta_{234} - \delta_3\delta_{1234}, \\
    &\delta_{24}\delta_{123} - \delta_{23}\delta_{124} + \delta_{12}\delta_{234} - \delta_2\delta_{1234}, 
    \delta_{14}\delta_{123} - \delta_{13}\delta_{124} + \delta_{12}\delta_{134} - \delta_1\delta_{1234}, \\
    &\delta_4\delta_{123} - \delta_3\delta_{124} + \delta_2\delta_{134} - \delta_1\delta_{234}, 
    \delta_{14}\delta_{23} - \delta_{13}\delta_{24} + \delta_{12}\delta_{34} - \delta_0\delta_{1234}, \\
    &\delta_4\delta_{23} - \delta_3\delta_{24} + \delta_2\delta_{34} - \delta_0\delta_{234}, 
    \delta_4\delta_{13} - \delta_3\delta_{14} + \delta_1\delta_{34} - \delta_0\delta_{134}, \\
    &\delta_4\delta_{12} - \delta_2\delta_{14} + \delta_1\delta_{24} - \delta_0\delta_{124}, 
    \delta_3\delta_{12} - \delta_2\delta_{13} + \delta_1\delta_{23} - \delta_0\delta_{123} \rangle.
\end{align*}
Therefore, the locus of spinors such that $X_\delta$ is singular is contained in $OG(4,9)$, that is, the locus of pure spinors. 
\qed

\medskip

\medskip 
Let us now turn to $OG(3,9)$. An important ingredient in the proof of Proposition \ref{OG27} was the observation that there exists a nonzero invariant quadratic form  $S^2\Delta\lra\CC$ (that defines the quadric $Q$, or the constant factor in (\ref{tenquadrics})), and also a nontrivial equivariant morphism $S^2\Delta\lra V_9$ also apparent from (\ref{tenquadrics}) (note that $\Delta\simeq\Delta^\vee)$. Seen as a quadratic map from $\Delta$ to $V_9$,
this map sends a generic spinor $\delta$ to a generic vector, in particular non isotropic, and this vector induces an orthogonal splitting $V_9=V_1\oplus V_8$. The stabilizer of $\delta$ gets embedded inside 
the corresponding $Spin_8\subset Spin_9$. Moreover the restriction of $\Delta$ to 
$Spin_8$ splits into $\Delta_+\oplus \Delta_-$, which up to an outer automorphism is equivalent to $V_8\oplus \Delta_+$. Then the restriction of $V_8$ to $Spin_7$ splits into $\CC\oplus V_7$, where the stable line must be generated by $\delta$. On the other hand, the restriction of $V_9$ to $Spin_8$ is $\CC\oplus V_8$, which under our outer automorphism becomes $\CC\oplus\Delta_-$, so that the restriction to $Spin_7$ yields $\CC\oplus\Delta_8$, where $\Delta_8$ is the spin representation of $\fso_7$. Therefore, in order to determine the  generic stabilizer of the action of $Spin_7$ on $OG(3,\Delta)$, we are reduced to $OG(3,\CC\oplus\Delta_8)$, hence to determine a subgroup of the generic stabilizer of  the action of $Spin_7$ on $OG(2,\Delta_8)$. 
We start with a Lemma. 

\begin{lemma}
The blow-up of $OG(2,\Delta_8)$ along $OG(2,V_7)$ is a $\QQ^4$-bundle over $\QQ^5$.
\end{lemma}

\proof Recall that $\PP(\Delta_8)$ contains $OG(3,V_7)$  as the invariant quadric hypersurface. So $OG(2,\Delta_8)$ parametrizes (spinorial) lines in 
$OG(3,V_7)$, which are conics for the Pl\"ucker embedding. 
Given a plane $P\in OG(2,V_7)$, the isotropic $3$-planes that contain $P$ are parametrized by such a conic. Hence an embedding $OG(2,V_7)\hookrightarrow OG(2,\Delta_8)$.  

Now, given a (spinorial) line $\delta$ in $OG(3,V_7)$, the $3$-planes it parametrizes must meet non trivially. Generically, they have a common line $\ell\subset V_7$. This defines a rational map $OG(2,\Delta_8)\dashrightarrow \QQ^5$, whose exceptional locus is $OG(2,V_7)$. 

To resolve this rational map, we consider the incidence variety $\cI$ parametrizing pairs $(\delta,\ell)$ as above. The first projection is an isomorphism outside 
the codimension two subvariety $OG(2,V_7)$, and a $\PP^1$-bundle over it; so it has to coincide with its blowup. The fiber over $\ell$ of the second projection is the variety of (spinorial) lines inside $OG(2,\ell^\perp/\ell)\simeq OG(2,5)$. But the latter is just $\PP^3$, and its lines are parametrized by $G(2,4)=\QQ^4$. \qed 

\medskip Using this identification, we  can deduce the stabilizer $\Gamma_P\subset Spin_7$ of a generic point $P\in OG(2,\Delta_8)$. Indeed, by the previous Lemma, belonging to $\Gamma_P$ is equivalent to fixing  an isotropic line  $\ell\subset V_7$, and then a line in $OG(2,\ell^\perp/\ell)\simeq\PP^3$. At this point it is convenient to use the isomorphism $\fso_5\simeq\fsp_4$, which identifies the spinorial representation of $\fso_5$ with the natural representation $V_4$ of $\fsp_4$. We are then reduced to computing the stabilizer in $\fsp_4$ of a generic plane $P\subset V_4$. Being generic $P$ is not isotropic and thus induces a decomposition $V_4=P\oplus P'$, and we deduce that the image of $\gamma_P=Lie(\Gamma_P)$ in $\fsp_4$ is a copy of $\fsl_2\times\fsl_2$. Moreover, the kernel of the restriction map to $\fsp_4$ is the subalgebra of $\fso_7$ that sends $\ell^\perp$ to $\ell$, and is easily checked to be a six-dimensional nilpotent subalgebra.  

Now, if $Q\in OG(3,\Delta_8\oplus \CC)$ is generic, its intersection with $\Delta_8$ is a generic point $P\in  OG(2,\Delta_8)$. Moreover, $Q$ is generated by $P$ and a vector $x+{\bf 1}$, for ${\bf 1}$ the generator of $\CC$ and $x$ a spinor in $P^\perp$, of norm $-1$, and uniquely defined up to $P$. The stabilizer $\Theta_Q$ of $Q$ in $Spin_7$ is thus the subgroup of the   stabilizer $\Gamma_P$ of $P$ that fixes $x$ modulo $P$. At the Lie algebra level, we have seen that $\gamma_P$ is an extension of  $\fsl(P)\times\fsl(P')$ by a $6$-dimensional nilpotent Lie algebra, for a decomposition $\Delta_4=P\oplus P'$ of the spin representation of $\fso_5$. But $\Delta_8$ restricted to $\fso_5$ is isomorphic to $\Delta_4\oplus \Delta_4$, and $P^\perp$ gets identified with $P'\oplus P'$. We deduce that the generic stabilizer in $\fsl(P)\oplus\fsl(P')$ is just $\fsl(P)$. Thus 
$\Theta_Q$ is an extension of  $\fsl(P)$ by a $6$-dimensional nilpotent Lie algebra. We have proved:

\begin{prop}
$OG(3,9)$ admits a quasi-homogeneous action of $Spin_7$, with generic stabilizer locally isomorphic to $SL_2\rtimes U_6$, where $U_6$ is unipotent.  
\end{prop}

This quasi-homogeneity property also holds for the corresponding spinorial Fano manifold:



\begin{prop}
$(OG(3,9), \cS)$ is a quasi-homogeneous $Spin_7$-Fano tenfold of index four and Picard number one.
\end{prop}

\noindent {\it Betti numbers}: $(1,1,2,3,3,4,3,3,2,1,1)$.

\proof
In order to prove that the action of $Spin_7$ is quasi-homogeneous, we compute the dimension of the generic stabilizer. For this it is convenient to upgrade to dimension $10$ and use the computation, required for 
$(OG(3,10), \cS_+)$, of the stabilizer of the $3$-plane $P=\langle e_1-f_2, e_3+f_4+e_5, e_4-f_5\rangle$. We deduce the stabilizer in $Spin_9$ by asking that a generic hyperplane $H=V_9$ containing $P$ is fixed. We chose $H=(e_3+e_5-f_5)^\perp$ and obtained the required dimension $11$ for the stabilizer.
\qed 

\begin{prop}\label{rigidity}
$(OG(3,9), \cS)$ is locally rigid.
\end{prop}

\proof Let $X$ be the zero locus of a general section $s \in H^0(OG(3,9),\cS)$. For simplicity, set $OG := OG(3,9)$. The normal exact sequence is
$$0\lra TX\lra T(OG)_{|X}\lra \cS_{|X}\lra 0.$$
In order to verify conditions (\ref{c1}), (\ref{c11}) and (\ref{c2}), we compute the Koszul complex associated to the section $s$. Since $\cS$ has rank $2$ and determinant $ \cO_{OG}(1)$, it takes the following form:
\[
0\lra \cO_{OG}(-1) \lra  \cS^\vee \lra \cO_{OG}\lra \cO_X \lra 0.
\]
The tangent bundle of $OG$ can be represented as an extension 
\[
0\lra Hom(\cU,\cU^\perp/\cU)\lra T(OG)\lra \bigwedge^2\cU^\vee\lra 0.
\]
In terms of irreducible vector bundles, we have $\cU^\vee = \cE_{\omega_1}$, $\bigwedge^2\cU^\vee = \cE_{\omega_2}$ and $Hom(\cU,\cU^\perp/\cU) = \cU^\vee\otimes \mathrm{Sym}^2 \cS(-1) = \cE_{\omega_1 -\omega_3 + 2 \omega_4}$.

$$\dynkin[labels={\cU^\vee=\cE_{\omega_1}, , , \cS=\cE_{\omega_4}}, edge length = 1cm] B{oo*o}$$

Therefore, to compute the cohomology of $T(OG)_{|X}$, we need to tensor the Koszul complex by $T(OG)$. In order to compute the cohomology of each term, we can consider the twisted Koszul complex by $\bigwedge^2\cU^\vee$ and by $\cU^\vee\otimes \mathrm{Sym}^2 \cS(-1)$. For $\bigwedge^2\cU^\vee$, we obtain the following list of homogeneous vector bundles, all irreducible:
\begin{itemize}
    \item $\bigwedge^2\cU^\vee(-1) = \cE_{\omega_2-\omega_3}$;
    \item $\bigwedge^2\cU^\vee\otimes \cS^\vee = \cE_{\omega_2-\omega_3+\omega_4}$;
    \item $\bigwedge^2\cU^\vee =\cE_{\omega_2} $.
\end{itemize}
Bott's theorem implies that the only non-vanishing cohomology is $$H^0\left(OG,\bigwedge^2\cU^\vee\right) = \bigwedge^2V_{9}^\vee.$$

On the other hand, using Bott's theorem again, we observe that every term in the Koszul complex twisted by $\cU^\vee\otimes \mathrm{Sym}^2 \cS(-1)$ is acyclic. We can thus deduce the cohomology of the terms in the Koszul complex twisted by $T(OG)$. From the previous discussion, the only term with non-trivial cohomology is $T(OG)$ in degree zero: 
\[
H^0(X,T(OG)_{|X}) \simeq \bigwedge^2V_{9}^\vee \simeq H^0(OG, T(OG)).
\]
We can conclude that conditions (\ref{c1}) and (\ref{c11}) are satisfied.
For condition (\ref{c2}), we twist the Koszul complex by $\cS = \cE_{\omega_4}$. In this case, the only non-vanishing cohomologies are:
\begin{itemize}
\item $H^0\left(OG,\cS\otimes \cS^\vee\right) = \mathbb{C}$;
\item $H^0\left(OG,\cS\right) = \Delta$.
\end{itemize}
So from the twisted Koszul complex we deduce the following exact sequence:
\[
0 \lra \mathbb{C}s \lra \Delta \lra H^0\left(X,\cS_{|X}\right) \lra 0,
\]
which means that condition (\ref{c2}) does also hold.
Arguing as explained in Section \ref{locRig}, we conclude that $H^1(X,TX) = 0$.
\qed

\medskip \noindent {\it Remark}. As we already mentioned in the introduction, the proof of local rigidity follows these exact same steps in all other cases except for Propositions \ref{rigidityOG311} and \ref{rigidityOG312}, where a more careful analysis is required.

\begin{prop}$OG(2,10)$ and $(OG(2,10), \cS_+)$ both admit quasi-homogeneous actions of $(\mathbb{G}_m\times Spin_7)\rtimes \Delta_8$. 
\end{prop}

\noindent {\it Betti  numbers:} $(1,1,2,3,3,3,3,2,1,1)$. 
\proof According to \cite[Proposition 31]{sk}, the generic stabilizer of the action of $\CC^*\times Spin_{10}$ on the $16$-dimensional half-spin representation $\Delta_+$ is locally isomorphic to $(\CC^*\times Spin_7)\rtimes \Delta_8$.  We use the explicit description of this generic stabilizer, or at least its Lie algebra $\fs$, as the stabilizer of the spinor $\delta=1+e_1e_2e_3e_4$, to compute the Lie algebra $\fs'$ (resp. $\fs''$) of the subgroup that also stabilizes the orthogonal plane $\langle e_5,f_1\rangle$ (resp. $\langle e_1,f_4\rangle$, that belongs to $(OG(2,10), \cS_+)$). The result of the computation is that $\fs'$ (resp. $\fs''$) 
has codimension $13$ (resp. $9$) in $\fs$, hence an open orbit.  
\qed 

\medskip \noindent {\it Remark.}
Pushing the 
computation one step further, we check that the generic stabilizer of the first action is locally isomorphic to $GL_3\rtimes U_7$, with $U_7$ unipotent of dimension $7$. Similarly, the generic stabilizer of the second action is locally isomorphic to  $(SL_2\times {\mathbb G}_m^2)\rtimes U_{15}$,
with $U_{15}$ unipotent of dimension $15$.

\begin{prop}
$(OG(2,10), \cS_+\oplus\cS_+)=G_2/P_2$
\end{prop}

\proof According to \cite[Proposition 31]{sk}, the generic stabilizer of the action of $GL_2\times Spin_{10}$ on  $\CC^2\otimes\Delta_+$ is locally isomorphic to $SL_2\times G_2$. Therefore, the generic stabilizer of $Spin_{10}$ acting on $G(2,\Delta_+)$ contains a copy of $G_2$. This implies that $(OG(2,10), \cS_+\oplus\cS_+)$ is a Fano fivefold of index three with an action of $G_2$. The classification of Mukai varieties (or, equivalently, the argument in the proof of Proposition \ref{G2P2}) ensures that this is $G_2/P_2$. \qed 

\begin{prop}\label{fl134}
$(OG(2,10), \cS_+\oplus\cS_-)=Fl_4$
\end{prop}

\proof The stabilizer of $\delta_+=1+e_{1234}\in\Delta_+$ in $\fso_{10}$ is described in \cite[(5.39)]{sk}. Observe that any element in this stabilizer kills $f_5$ and has its image contained in $f_5^\perp$. Note that $\delta_+$ is the sum (in a unique way) of two pure spinors, namely $1$ and $e_{1234}$, which correspond to the maximal isotropic spaces $\langle f_1,f_2,f_3,f_4,f_5\rangle$ and $\langle e_1,e_2,e_3,e_4,f_5\rangle$, respectively; their intersection is the line generated by $f_5$. 

Let us choose another generic spinor, now in $\Delta_-$, say $\delta_-=e_5+\zeta e_{12345}$, with $\zeta^2\ne 1$. Again it is the sum of the two pure spinors $e_5$ and $\zeta e_{12345}$, which correspond to the maximal isotropic spaces $\langle f_1,f_2,f_3,f_4,e_5\rangle$ and $\langle e_1,e_2,e_3,e_4,e_5\rangle$, respectively; their intersection is the line generated by $e_5$. Therefore, any element in the stabilizer of $\delta_-$ has to kill $e_5$, and its image must be contained in $e_5^\perp$. 

With a little extra work, we find that the intersection of the two stabilizers is the space of matrices of the form
$$\begin{pmatrix}
a_1& a_{12}& a_{13}& a_{14}& 0& 0&0&0&0&0 \\
a_{21}& a_2& a_{23}& a_{24}& 0& 0&0&0&0&0 \\
a_{31}& a_{32}& a_3& a_{34}& 0& 0&0&0&0&0 \\
a_{41} & a_{42}& a_{43}& a_4& 0& 0&0&0&0&0 \\
 0&0&0&0&0& 0&0&0&0&0 \\
 0&0&0&0&0& -a_{1}& -a_{21}& -a_{31}& a_{41}& 0\\
 0&0&0&0&0& -a_{12}& -a_{2}& -a_{32}& a_{42}& 0\\
 0&0&0&0&0& -a_{13}& -a_{23}& -a_{3}& a_{43}& 0\\
 0&0&0&0&0& -a_{14}& -a_{24}& -a_{34}& a_{4}& 0\\
 0&0&0&0&0& 0&0&0&0&0  
 \end{pmatrix}$$
 with $a_1+a_{2}+a_{3}+a_{4}=0$; this is just a copy of $\fsl_4$. Now, recall that $OG(2,10)$ can be identified with the space of matrices in $\fso_{10}$, up to scalar, which have rank two and isotropic image.
 Clearly the matrix above has rank two when the matrix $A=(a_{ij})$ has rank one, which means that we can find two nonzero vectors $u,v$
 in $\CC^4$ such that $a_{ij}=u_iv_j$ for all $i,j$.
 Then the image of $A$ is $\sum u_ie_i$ and the image of its transpose is $\sum v_jf_j$; these two vectors  generate an isotropic plane when $\langle u,v\rangle =0$. This means that we get a subvariety of 
$X=(OG(2,10), \cS_+\oplus\cS_-)$ isomorphic to $Fl_4$.
Since  the latter has the correct dimension and the former is connected (a computation based on the Koszul complex gives $h^0(\cO_X)=1$), they must be equal. 
\qed 

\medskip\noindent 
{\it Remark.} We have proved in passing that 
$$\PP(\Delta_+)\times \PP(\Delta_-)\supset Spin_{10}/\Gamma$$
as an open orbit, where $\Gamma$ is locally isomorphic to $SL_4=Spin_6$. Since $\Delta_+$ and $\Delta_-$ are dual one to the other, we can expect the complement of this open orbit to be the invariant quadric defined by this duality.  

\begin{prop}
$(OG(3,10), \cS_+)$ is a Fano manifold of dimension $13$, index $5$ and Picard number one, with a quasi-homogeneous action of $(\mathbb{G}_m\times Spin_7)\rtimes\Delta$. \end{prop}

\noindent{\it Betti numbers:} $(1,1,3,4,6,7,8,8,7,6,4,3,1,1)$.

\proof We use the explicit form of the Lie algebra of $(\mathbb{G}_m\times Spin_7)\rtimes\Delta$, given in 
\cite[(5.39)]{sk} as the stabilizer $\fs$ in $\fso_{10}$, of the generic spinor $\delta= 1+e_{1234}$. This spinor is killed by 
the Pl\"ucker representative $\omega_P$ of the isotropic three-plane $P=\langle e_1-f_2, e_3+f_4+e_5, e_4-f_5\rangle$. An explicit computation shows that the stabilizer of $P$ has codimension $13$ in $\fs$. This  coincides with the dimension of $(OG(3,10), \cS_+)$, whose orbit must therefore be dense. \qed 

\begin{prop}
$(OG(3,10), \cS_+\oplus \cS_+)$ is a Fano manifold of dimension $11$, index $4$ and Picard number one, with a quasi-homogeneous action of $G_2\times SL_2$. \end{prop}

\noindent  {\it Betti numbers:} $(1,1,3,4,6,-,-,6,4,3,1,1)$.

\proof It is proved in \cite[Proposition 32]{sk} that the generic stabilizer for the action of $Spin_{10}\times GL_2$ on $\Delta_+\otimes \CC^2$ is locally isomorphic to $G_2\times SL_2$.
More precisely they compute the stabilizer of the pair $\delta_+=1+e_{1234}$, $\delta'_+=e_{15}+e_{2345}$.
The Pl\"ucker representative of the isotropic $3$-plane $P=\langle e_5-f_1, e_5-f_2, e_1+e_2+f_5\rangle$ kills both $\delta_+$ and $\delta'_+$, so $[P]$ is a point of $X$. Using \cite[(5.40)]{sk} we find that its  stabilizer is six dimensional, so its orbit has dimension $17-6=11$ and must be dense in $X$.
\qed 

\medskip \noindent {\it Remark}. We can be more precise: with the notations of 
\cite[(5.40)]{sk}, the Lie algebra $\gamma$ of the stabilizer $\Gamma$ of 
$[P]$ in $G_2\times SL_2$ is given by the equations 
$$a_{13}=a_{14}=a_{23}=a_{24}=a_{31}=a_{32}=a_{41}=a_{42}=0, $$
$$d_{12}=a_{21}, \quad d_{21}=a_{12}-a_{21}-a_2, \quad  2d_{11}=-a_2-2a_{21}.$$
The kernel of the projection of $\gamma$ to $\fsl_2$ is the same as the kernel of its projection to $\fgl(P)$, and is the  copy of $\fsl_2$ in $\fg_2$ corresponding to the coefficients $a_{34}, a_{43}, a_4=-a_3$. This implies that $\gamma\simeq\fsl_2\times\fsl_2$,  that 
$G_2\cap\Gamma=\Gamma_0$ is locally isomorphic to 
$SL_2$, and that $X$ is also a compactification of $G_2/\Gamma_0$. Comparing \cite[(5.42)]{sk} with \cite[Example 30, (1.8)]{sk}, we see that $\gamma_0$
is a copy of $\fsl_2$ generated by a short root of $\fg_2\simeq\fsl_3\oplus S^2\CC^3$.

\begin{prop}
$(OG(3,10), \cS_+\oplus \cS_-)$ is a Fano manifold of dimension $11$, index $4$ and Picard number one, with a quasi-homogeneous action of $SL_4$. \end{prop}

\noindent  {\it Betti numbers:} $(1,1,3,4,-,-,-,-,4,3,1,1)$.

\proof We compute the stabilizer of an element of $(OG(3,10), \cS_+\oplus \cS_-)$, which we choose to be the $3$-plane $P=\langle e_1+f_2,e_3-f_4, e_3-e_5\rangle$; a straightforward computation shows that $\omega_P$ kills both $\delta_+$ and $\delta_-$. 

From the description of $\fsl_4$ as the common stabilizer to  $\delta_+$ and $\delta_-$ given in the proof of Proposition \ref{fl134}, it is easy to describe the subalgebra that preserves $P$; it is given by the matrices $A$ of the form 
$$A=\begin{pmatrix}
0&\alpha & * &* \\ 0&\beta& *&* \\ 0&0&-\beta & \alpha\\
0&0&0&0 \end{pmatrix}$$
with respect to the basis $e_3,e_1,e_2,e_4$.
In particular, it has dimension $6$, so the orbit of $\omega_P$ has dimension $9$ and must be dense in 
$(OG(3,10), \cS_+\oplus \cS_-)$.\qed 

\begin{prop}
$(OG(5,10)_-, \cS_+)=\QQ^5$
\end{prop}

\proof By adjunction, $(OG(5,10)_-, \cS_+)$ is a Fano fivefold of index $5$ -- hence a quadric. \qed 

\begin{prop}
$(OG(2,11), \cS)=Fl_5$
\end{prop}

\proof According to \cite[Proposition 39]{sk}, the generic stabilizer of the action of $\CC^*\times Spin_{11}$ on the $32$-dimensional spin representation $\Delta$ is locally isomorphic to $SL_5$. Moreover, the corresponding embedding of $\fsl_5$ into $\fso_{11}$ is the standard one, in the sense that it preserves a decomposition of $\CC^{11}$ as $E\oplus F\oplus L$ where $E,F$ are maximal isotropic (hence $5$-dimensional) and $L$ is the orthogonal line to the hyperplane $E\oplus F$. Then, exactly as we observed in the proof of Proposition \ref{fl134}, the adjoint variety of $\fsl_5$ embeds in the adjoint variety of $\fso_{11}$, and the claim follows.  \qed 

\begin{prop}\label{OG311}
$(OG(3,11), \cS)$ is a Fano manifold of dimension $14$, index $5$ and Picard number one, with a quasi-homogeneous action of $SL_5$.
\end{prop}

\noindent  {\it Betti numbers:} $(1,1,2,3,5,-,-,-,-,-,5,3,2,1,1)$

\proof The (generic) stabilizer of $\delta=1+e_{123456}$ in $Spin_{11}$ is computed in the proof of \cite[Proposition 38]{sk}, where they embed $\CC^{11}$ in $\CC^{12}$ as the hyperplane orthogonal to $e_6-f_6$; they prove its Lie algebra is isomorphic to $\fsl_5$ by checking that it must kill both $e_6$ and $f_6$. 

We need to modify the $3$-plane we chose in the previous proof since it was not contained in the hyperplane $\CC^{11}$: we replace it by  $P'=\langle e_1-f_1-e_6-f_6, e_2+f_5, e_3+f_4\rangle$, which is easily check to be isotropic, orthogonal to $e_6-f_6$, and to kill $\delta$. A straightforward computation shows that its stabilizer in $\fsl_5$ is the space of matrices of the form
$$\begin{pmatrix}
0&0&0 &\beta_1&\beta_2 \\
\beta_2& \alpha_{21}& \alpha_{22}& \gamma_{11}& \gamma_{12}\\
\beta_1& \alpha_{11}& \alpha_{12}& \gamma_{21}& \gamma_{22}\\
0&0&0 &- \alpha_{12}& -\alpha_{22}\\
0&0&0 &- \alpha_{11}& -\alpha_{21}
\end{pmatrix}$$
In particular its dimension is $10$, so the $SL_5$-orbit of $P'$ has dimension $24-10=14$ and must be open in $X$. \qed

\medskip\noindent {\it Remark.} The previous computation shows that the generic stabilizer of the action of $SL_5$ on $(OG(3,11),\cS)$ is locally isomorphic to $GL_2\rtimes U_6$, with $U_6$ unipotent and acted on by $GL_2$ through the sum of the adjoint and the natural representations.

\begin{prop}\label{rigidityOG311}
$(OG(3,11), \cS)$ is locally rigid.
\end{prop}

\proof Let $X$ be the zero locus of a general section $s \in H^0(OG(3,11),\cS)$ and set $OG := OG(3,11)$. We start as in the proof of Proposition \ref{rigidity}. Since $\cS$ has rank $4$ and determinant $\cO_{OG}(2)$, the Koszul complex takes the following form:
\[
0\lra \cO_{OG}(-2) \lra \cS(-2)  \lra \bigwedge^2\cS^\vee \lra \cS^\vee \lra \cO_{OG}\lra \cO_X \lra 0.
\]
The tangent bundle of $OG$ can be represented as an extension 
\[
0\lra Hom(\cU,\cU^\perp/\cU)\lra T(OG)\lra \bigwedge^2\cU^\vee\lra 0.
\]
In the following, we denote $\mathcal{R} = \cE_{\omega_4}$, which satisfy $\bigwedge^2\cS = \mathcal{R} \oplus \cO(1)$. 

In terms of irreducible vector bundles, we have $\cU^\vee = \cE_{\omega_1}$, $\bigwedge^2\cU^\vee = \cE_{\omega_2}$ and $Hom(\cU,\cU^\perp/\cU) = \cU^\vee\otimes \mathcal{R}(-1) = \cE_{\omega_1 -\omega_3 + \omega_4}$.
Therefore, to compute the global sections of $T(OG)_{|X}$, we need to tensor the Koszul complex by $\bigwedge^2\cU^\vee$ and by $\cU^\vee\otimes \mathcal{R}(-1)$. 

For $\bigwedge^2\cU^\vee$, we obtain the following list of completely reducible vector bundles:
\begin{itemize}
    \item $\bigwedge^2\cU^\vee(-2) = \cE_{\omega_2 - 2\omega_3}$;
    \item $\bigwedge^2\cU^\vee \otimes \cS(-2) = \cE_{\omega_2 - 2\omega_3 + \omega_5}$;
    \item $\bigwedge^2\cU^\vee(-1) \oplus \bigwedge^2\cU^\vee \otimes \mathcal{R}(-2) = \cE_{\omega_2 - \omega_3} \oplus \cE_{\omega_2 - 2\omega_3 + \omega_4}$;
    \item $\bigwedge^2\cU^\vee \otimes \cS(-1) = \cE_{\omega_2 - \omega_3 + \omega_5}$;
    \item $\bigwedge^2\cU^\vee = \cE_{\omega_2}$.
\end{itemize}
Bott's theorem implies that the only non-vanishing cohomologies are:
\begin{itemize}
\item $H^1\left(OG,\bigwedge^2\cU^\vee\otimes \bigwedge^2\cS^\vee\right) = \mathbb{C}$;
\item $H^0\left(OG,\bigwedge^2\cU^\vee\right) = \bigwedge^2V_{11}^\vee$.
\end{itemize}
Similarly, for $\cU^\vee(-1) \otimes \mathcal{R}$ one deduces that the only non-trivial cohomology is $H^2\left(OG,\cU^\vee(-1) \otimes \mathcal{R}\otimes \bigwedge^2\cS^\vee\right) = \mathbb{C}$. In particular, we can compute the cohomology of the terms in the Koszul complex twisted by $T(OG)$. The only ones with non-trivial cohomology are $T(OG)\otimes \bigwedge^2\cS^\vee$, which satisfies
\[
0 \lra H^1\left(T(OG)\otimes \bigwedge^2\cS^\vee\right) \lra  \mathbb{C} \xrightarrow{\phi} \mathbb{C} \lra H^2\left(T(OG)\otimes \bigwedge^2\cS^\vee\right)  \lra 0,
\]
and $ T(OG)$, whose global sections are $\bigwedge^2V_{11}^\vee$. We have to take into account the behaviour of the map $\phi$. If $\phi$ is not zero, then $T(OG)\otimes \bigwedge^2\cS^\vee$ is acyclic and conditions (\ref{c1}) and (\ref{c11}) follow. If $\phi= 0$, we get an exact sequence $$ 0\lra \CC\lra H^0(T(OG))\lra H^0(T(OG)_{|X})\lra \CC \lra 0. $$ 
Let $\omega\in H^0(T(OG))= \wedge^2V_{11}^\vee$ a non-zero element that vanishes as a vector field on $X$. Under the action of $SL_5$, $\wedge^2V_{11}^\vee= E_5\oplus F_5\oplus \CC$ decomposes as
\[
\wedge^2V_{11}^\vee \simeq \wedge^2 E_{5} \oplus \wedge^2 E_{5}^\vee \oplus \mathfrak{sl}(E_5) \oplus \CC \oplus E_5 \oplus E_5^\vee .
\]
By $SL_5$-invariance, $\omega$ must be the generator of the one-dimensional summand which, in $\wedge^2V_{11}^\vee$ corresponds to $\mathrm{id}_{E_5} - \mathrm{id}_{F_5}$. Since $\omega$ vanishes on $X$, it does stabilize every three dimensional plane in $X$. Since this is not true for the element $P'$ introduced in the proof of Proposition \ref{OG311}, we get a contradiction.

\medskip
The verification of conditions (\ref{c2}) works as in the proof of Proposition \ref{rigidity}.
\qed 

\begin{prop}
$(OG(2,12), \cS_+)=Fl_6$
\end{prop}

\proof A generic spinor in $\Delta_+$ is $\delta=1+e_{123456}$, which is the sum of two pure spinors corresponding to the maximal isotropic spaces $E$ and $F$. The stabilizer of $[\delta]$ in $Spin_{12}$ fixes the pair $(E,F)$, so its connected component fixes both $E$ and $F$. Hence it maps naturally to $GL(E)$ and $GL(F)$ with the same determinant since it must multiply $1$ and $e_{123456}$ by the same factor. But these two actions are transpose one to the other, so in fact the connected component of the stabilizer naturally identifies to $SL(E)$, as checked in the proof of \cite[Proposition 37]{sk}. 

The adjoint representation of $Spin_{12}$,  restricted to $SL_6$,  decomposes as 
$$\fso_{12}=\fgl_6\oplus \wedge^2E\oplus \wedge^2E^\vee .$$
In particular, the adjoint variety $OG(2,V_{12})$ of $\fso(V_{12})$ contains the
adjoint variety $Fl_6$ of $\fsl_6$: to a flag $(\ell\subset H=\phi^\perp\subset E)$ we simply associate the orthogonal plane $\langle \ell,\phi\rangle $ in $V_{12}$. Since $\phi(\ell)=0$, it is easy to see that this plane belongs to 
$(OG(2,12), \cS_+)$, and the conclusion follows.
\qed

\begin{prop}\label{OG312}
$(OG(3,12), \cS_+)$ is a  Fano manifold of dimension $17$, index $6$ and Picard number one, with a quasi-homogeneous action of $SL_6$.
\end{prop}

\noindent  {\it Betti numbers:} $(1,1,2,4,6,8, -,-,-,-,-,-,8,6,4,2,1,1)$.

\proof 
We claim that the $3$-plane $P=\langle e_1+f_6, e_2+f_5, e_3+f_4\rangle$ belongs to $X$, as a straightforward application of Lemma \ref{kills}. We can easily compute its stabilizer in $\fsl_6$ to be the space of matrices of the form
\begin{align}\label{stabSL6}
\begin{pmatrix}
a_{11} &a_{12} &a_{13} & b_{11} &b_{12} &b_{13} \\
a_{21} &a_{22} &a_{23} & b_{21} &b_{22} &b_{23} \\
a_{31} &a_{32} &a_{33} & b_{11} &b_{12} &b_{13} \\
0&0&0&-a_{11} &-a_{21} &-a_{31}\\
0&0&0&-a_{12} &-a_{22} &-a_{32}\\
0&0&0&-a_{13} &-a_{23} &-a_{33}
\end{pmatrix}
\end{align}
This has dimension $18$, so the $SL_6$-orbit of $P$ has dimension
$35-18=17$, hence must be open in $X$. \qed 

\medskip \noindent {\it Remark.} The previous computation shows that the generic stabilizer of the action of $SL_6$ on $(OG(3,12), \cS_+)$ is locally isomorphic to $GL_3\rtimes U_9$, where $GL_3$ acts on the unipotent $U_9$ through the adjoint representation.

\begin{prop}\label{rigidityOG312}
$(OG(3,12), \cS_+)$ is locally rigid.
\end{prop}

\proof Let $X$ be the zero locus of a general section $s \in H^0(OG(3,12),\cS_+)$ and set $OG := OG(3,12)$. We start as in the proof of Proposition \ref{rigidity}. Since $\cS_+$ has rank $4$, the Koszul complex takes the following form
\[
0\lra \cO_{OG}(-2) \lra \cS_+(-2)  \lra \bigwedge^2\cS_+^\vee \lra \cS_+^\vee \lra \cO_{OG}\lra \cO_X \lra 0.
\]
The tangent bundle of $OG$ can be represented as an extension 
\[
0\lra Hom(\cU,\cU^\perp/\cU)\lra T(OG)\lra \bigwedge^2\cU^\vee\lra 0.
\]
In terms of irreducible vector bundles, we have $\cU^\vee = \cE_{\omega_1}$, $\bigwedge^2\cU^\vee = \cE_{\omega_2}$ and $Hom(\cU,\cU^\perp/\cU) = \cU^\vee\otimes \bigwedge^2\cS_+(-1) = \cE_{\omega_1 -\omega_3 + \omega_4}$.

$$\dynkin[labels={\cU^\vee, \wedge^2\cU^\vee,, ,\cS_+=\cE_{\omega_6}, \cS_-=\cE_{\omega_5}}, edge length = 1cm] D{ooo*oo}$$

\medskip
Therefore, to compute the global sections of $T(OG)_{|X}$, we need to tensor the Koszul complex by $\bigwedge^2\cU^\vee$ and by $\cU^\vee\otimes \bigwedge^2\cS_+(-1)$. 
For $\bigwedge^2\cU^\vee$, we obtain the following list of irreducible vector bundles:
\begin{itemize}
    \item $\bigwedge^2\cU^\vee(-2) = \cE_{\omega_2 - 2\omega_3}$;
    \item $\bigwedge^2\cU^\vee \otimes \cS_+(-2) = \cE_{\omega_2 - 2\omega_3 + \omega_6}$;
    \item $\bigwedge^2\cU^\vee \otimes \bigwedge^2\cS_+ (-2)  = \cE_{\omega_2 - 2\omega_3 +\omega_4} $;
    \item $\bigwedge^2\cU^\vee \otimes \cS_-(-1) = \cE_{\omega_2 - \omega_3 + \omega_5}$;
    \item $\bigwedge^2\cU^\vee = \cE_{\omega_2}$.
\end{itemize}
Bott's theorem implies that the only non-vanishing cohomologies are:
\begin{itemize}
\item $H^1\left(OG,\bigwedge^2\cU^\vee\otimes \bigwedge^2\cS_+(-2)\right) = \mathbb{C}$;
\item $H^0\left(OG,\bigwedge^2\cU^\vee\right) = \bigwedge^2V_{12}^\vee$.
\end{itemize}
Similarly, for $\cU^\vee\otimes \bigwedge^2\cS_+(-1)$ one deduces that the only non-trivial cohomology is $H^2\left(OG,\cU^\vee\otimes \bigwedge^2\cS_+(-1)\otimes \bigwedge^2\cS_+^\vee\right) = \mathbb{C}$. In particular, we can compute the cohomology of the terms in the Koszul complex twisted by $T(OG)$. The only one with non-trivial cohomology is $T(OG)\otimes \bigwedge^2\cS_+^\vee$, for which we get a long exact sequence 
\[
0 \lra H^1\left(T(OG)\otimes \bigwedge^2\cS_+^\vee\right) \lra  \mathbb{C} \xrightarrow{\phi} \mathbb{C} \lra H^2\left(T(OG)\otimes \bigwedge^2\cS_+^\vee\right)  \lra 0,
\]
and $T(OG)$, whose space of global sections is $\bigwedge^2V_{12}^\vee$. We have to take into account the behaviour of the map $\phi$. If $\phi$ is not zero, then $T(OG)\otimes \bigwedge^2\cS_+^\vee$ is acyclic and conditions (\ref{c1}) and (\ref{c11}) follow. If $\phi= 0$ we get an exact sequence $$0\lra\CC\lra H^0(T(OG))\lra H^0(T(OG)_X)\lra \CC \lra 0. $$
Let $\omega\in H^0(T(OG))= \wedge^2V_{12}^\vee$ a non-zero element that vanishes as a vector field on $X$. Under the action of $SL_6$, $\wedge^2V_{12}^\vee= \wedge^2(E\oplus E^\vee)$ decomposes as
\[
\wedge^2V_{12}^\vee \simeq \wedge^2 E \oplus \wedge^2E^\vee \oplus \mathfrak{sl}(E) \oplus \CC 
\]
By $SL_6$-invariance, $\omega$ must be the generator of the one-dimensional summand which, in $\wedge^2V_{12}^\vee$ corresponds to $\mathrm{id}_{E} - \mathrm{id}_{F}$. Since $\omega$ vanishes on $X$, it does stabilize every three dimensional plane in $X$. But this is not true for the element $P$ introduced in the proof of Proposition \ref{OG312}, and we get a contradiction.

\smallskip
Condition (\ref{c2}) is verified as in the proof of Proposition \ref{rigidity}.
\qed

\begin{prop}
$(OG(4,12), \cS_+)$ is a quasi-homogeneous Fano manifold of dimension $20$, index $6$ and Picard number one.
\end{prop}

\noindent  {\it Betti numbers:} $(1,1,3,4,7,9,13,-,-,-,-,-,-,-,13,9,7,4,3,1,1)$. 

\proof We proceed as before by identifying the stabilizer in $\fsl_6$ of a $4$-plane in $X$. 
We choose 
$$Q=\langle e_1+f_6, e_2+f_5, e_3+f_4, e_4+e_5+e_6-f_1-f_2-f_3\rangle ,$$
and check that its stabilizer is again a space of matrices of the form (\ref{stabSL6}), with the extra conditions that the matrix with entries $a_{ij}-b_{ij}$ has its row sums and column sums all equal to zero. These gives $5$ independent conditions and shows that the stabilizer of $Q$ has dimension $13$. So its orbit has dimension $35-13=22$, and must be open in $X$. \qed 

\begin{prop}
$(OG(6,12)_-, \cS_+)=G(3,6)$
\end{prop}

\proof First observe that $(OG(6,12)_-, \cS_+)$ is a priori a Fano ninefold of index $6$. It is acted on 
by the stabilizer of a generic $\delta\in\Delta_+$, locally isomorphic to $SL_6$. The natural representation of $SO_{12}$ gets decomposed into $E\oplus F$, with $F\simeq E^\vee$. This allows to embed $G(3,6)$ into $OG(6,12)$ by mapping $L\in G(3,E)$ to $P=L\oplus L^\perp$. Since these spaces have odd dimensional intersection with $E$, they belong to the other family $OG(6,12)_-$. A Pl\"ucker representative of $P$ is $\omega_P=\omega_{L^\perp}\wedge \omega_L$. The action of  $\omega_L$ on the spinor $\delta$ sends it to  $\omega_L\in\wedge^3L$, which is clearly killed by contraction with  $\omega_{L^\perp}$. Therefore, $G(3,6)$ is embedded into $(OG(6,12)_-, \cS_+)$, and this concludes the proof. \qed  

\begin{prop}
$(OG(2,14), \cS_+)=G_2/P_2\cup G_2/P_2$
\end{prop}\label{doubleG2P2}

\noindent {\it Remark.} This is the "doubled" version of Proposition \ref{G2P2}.

\proof According to \cite[Proposition 40]{sk}, the generic stabilizer of the action of $\CC^*\times Spin_{14}$ on the $64$-dimensional half-spin representation $\Delta_+$ is locally isomorphic to $G_2\times G_2$. Therefore, the Lie algebra $\fso_{14}$ contains two copies of $\fg_2$ and the adjoint variety $OG(2,14)$ contains two copies of $G_2/P_2$. If we can prove that they are contained in 
$(OG(2,14), \cS_+)$, we will be done. 

This can be checked with the help of the {\it multiplicative double point property} of \cite[Proposition 2.2.1]{am}: given a generic $\delta\in\Delta_+$, there is a unique decomposition $V_{14}\simeq V_7\oplus V'_7$ compatible with the action of its stabilizer, and yielding an identification of $\Delta_+$ with $\Delta_8\otimes \Delta'_8$ such that $\delta\simeq \chi\otimes\chi'$ for some generic $\chi\in\Delta_8$ and $\chi'\in \Delta'_8$. Here $\Delta_8$ denotes the spin representation of $\fso_7$. The isotropic lines in $\wedge^2V_7$ that kill $\chi$ are parametrized by a copy of $G_2/P_2$ (Proposition \ref{G2P2}), and since they also kill $\delta$, we are done.  \qed

\begin{prop}\label{3,14}
$(OG(3,14), \cS_+)$ is a Fano manifold of dimension $19$, index $6$ and Picard number one, with a quasi-homogeneous action of $G_2\times G_2$.
\end{prop}

\noindent{{\it Betti numbers:} $(1,1,2,3,6,7,10,12,15,15,15,15,12,10,7,6,3,2,1,1)$.}

\proof We use the explicit computations of the proof of \cite[Proposition 40]{sk}, where the stabilizer in $\CC^*\times Spin_{14}$ of the (generic) spinor
$$\delta=1+e_{1237}+e_{4567}+e_{123456}$$
is explicitly identified to $G_2\times G_2$, at least at the level of their Lie algebras.  We claim that the isotropic $3$-plane 
$$P=\langle e_1+e_4+f_1-f_4, e_2+e_5+f_2-f_5, e_3+e_6+f_3-f_6 \rangle$$
defines a point $X$ -- we let to the reader the boring task to check that $\omega_P$ kills $\delta_+$. 
Inside the Lie algebra stabilizer of $\delta_+$ described by \cite[(5.49)]{sk}, the notations of which we follow, we identify the stabilizer  $\fs$ of $P$ as the sub-Lie algebra given by the relations
$$x_{11}=x_{22}=x_{33}=y_{11}=y_{22}=y_{33}=0,$$
$$u+a=v+b=w+c=\lambda+d=\mu+e=\nu+f=0,$$
$$x_{12}+x_{21}=x_{13}+x_{31}=x_{23}+x_{32}=0,$$
$$y_{12}+y_{21}=y_{13}+y_{31}=y_{23}+y_{32}=0,$$
$$x_{12}-y_{12}=c-f, \quad x_{13}-y_{13}=e-b, \quad 
x_{23}-y_{23}=a-d.$$
In particular this stabilizer has dimension $9$, hence the orbit of $P$ has dimension $2\times 14-9=19$ and must be dense in $X$. 
 \qed 
 
\medskip \noindent {\it Remark}. With some extra work we can identify the Lie algebra stabilizer $\fs$ of $P$ as follows. From the relations above we can and will express all the indeterminates in terms of $a,b,c,d,e,f,x_{12},x_{13},x_{23}$. We first observe that $\fs$ acts on $P$ by the following matrix, in the defining basis:
 $$\begin{pmatrix}
 0&x_{12}-c& x_{13}+b \\ c-x_{12}& 0 & x_{23}-a\\
 -b-x_{13} & a-x_{23} & 0 
 \end{pmatrix}.$$
This generic skew-symmetric matrix generates a copy of $\fso_3\simeq\fsl_2$. Then we describe the subalgebra $\fs_0$ of $\fs$ that acts trivially on $P$, which means that it is subject to the extra conditions $x_{12}=c$, $x_{13}=-b$, $x_{23}=a$. We claim that inside $\fg_2\times \fg_2$, $\fs_0$ embeds as a product $\fs_1\times \fs_2$ of two subalgebras embedded inside the two copies of $\fg_2$, respectively. To see this, we use the identification of the stabilizer of $\delta_+$ with $\fg_2\times\fg_2$ provided by the matrices denoted $A_4, A_5$ in \cite[(5.49)]{sk}. For a matrix in $\fs_0$ defined by $(a,b,c,d,e,f)$, we get for example
$$A_4=\begin{pmatrix}
0&-2a&-2b&-2c&2a&2b&2c \\
a& 0 & c & -b & 0 & -c& b \\
b & -c &0 & a & c &0 & -a \\
c& b&-a& 0 & -b & a& 0\\ 
-a& 0 & -c & b & 0 & c& -b \\
-b & c &0 & -a & -c &0 & a \\
-c& -b&a& 0 & b & -a& 0 
\end{pmatrix}.$$
This matrix does not depend on $d,e,f$. Moreover, it we call $g_0,\ldots , g_6$ the basis vectors, it kills 
$g_1+g_4, g_2+g_5, g_3+g_6$, has its image contained in the span of $(2g_0, g_1-g_4, g_2-g_5, g_3-g_6)$, and acts on the space having these four vectors as basis by the matrix
$$2\begin{pmatrix}
0 & -a & -b& -c \\ a & 0 & c&-b\\
b&-c&0 &a\\ c& b&-a&0
\end{pmatrix}.$$
This is easily seen to give a copy of $\fsl_2$, acting on the sum of two copies of its natural representation. In particular $\fs_1\simeq\fsl_2$ and symmetrically $\fs_2\simeq\fsl_2$. We finally conclude that the stabilizer of $P$ in $G_2\times G_2$ is locally isomorphic to $SL_2\times SL_2\times SL_2$.

\medskip\noindent {\it Remark.} As we recalled in the proof of  Proposition \ref{doubleG2P2}, the action of $G_2\times G_2$ can be traced back from an orthogonal decomposition $V_{14}\simeq V_7\oplus V'_7$ into two copies of the vector representation of $G_2\subset Spin_7$. A maximal torus of $G_2$ acts on $V_7$ with weights $0, \pm\epsilon_1, \pm\epsilon_2, \pm\epsilon_3$, where the weight space of weight zero is not isotropic. We deduce that a maximal torus of $G_2\times G_2$ acts on $V_{14}$ with weights 
$\pm\epsilon_i, \pm\epsilon'_i$, $1\le i\le 3$, and $0$ with multiplicity $2$. The induced action on $G(3,14)$ therefore fixes finitely many points, plus a collection of projective lines parametrizing $3$-planes generated by two weight spaces of nonzero weight, and some weight zero vector. Since the zero weight space is not isotropic, only finitely many of those fixed points parametrize isotropic $3$-planes.  We conclude  that a maximal torus of $G_2\times G_2$ has finitely many fixed points in $OG(3,14)$, hence also, a fortiori, in $(OG(3,14),\cS_+)$. This implies, by the Bialynicki-Birula decomposition, that this variety has pure cohomology, and is a compactification of the affine space. This simplifies a lot the computation of the Betti numbers, which can also be deduced from the torus action on the tangent spaces at the fixed points.

\begin{prop}
$(OG(4,14), \cS_+)$ is a Fano manifold of dimension $26$, index $7$ and Picard number one, with an action of $G_2\times G_2$ of cohomogeneity at most one.
\end{prop}

\noindent{\it Betti numbers:} $(1,1,2,4,6,8,12,16,20,-,\ldots,-,20,16,12,8,6,4,2,1,1)$.

\proof Consider the isotropic $4$-space 
$$Q=\langle e_1+e_4+f_1-f_4, e_2+e_5+f_2-f_5, e_3+e_6+f_3-f_6, e_7 \rangle .$$
Using the previous computations, we check that the Lie algebra stabilizer of $Q$ in $\fg_2\times\fg_2$ is a copy of $\fsl_2$, so the orbit of $Q$ has codimension one. This implies that the generic orbit has codimension at most one, as claimed. \qed 

\medskip\noindent {\it Remark.} We have not been able to establish which of the two following possibilities does hold. Either there is an open orbit $(G_2\times G_2)/K$, with $K$ two-dimensional, and this open orbit must be  affine since it is the complement of an ample hypersurface (because the Picard number is one); moreover, by Richardson's theorem applied to the cone over $X$, the connected component of $K$ would be conjugate to a subgroup of $SL_2$, and the only possibility in dimension two would be a Borel subgroup. Or the generic orbit has codimension one in $X$ (and the generic stabilizer has Lie algebra $\fsl_2$). Extensive computations of stabilizers at  "random"  points, for which we always found three-dimensional stabilizers,  make the second hypothesis more likely than the first one.

\begin{prop}
$(OG(5,14), \cS_+)$ is a Fano manifold of dimension $28$, index $7$ and Picard number one, with an action of $G_2\times G_2$ of cohomogeneity at most one.
\end{prop}

\medskip\noindent {\it Betti numbers:} $(1,1,3,4,7,10,14,18,23,29,-,\ldots,-,29,23,18,14,10,7,4,3,1,1)$. 

\proof Consider the three-plane $P$ we found in the proof of Proposition \ref{3,14}. Let $R$ be the isotropic space spanned by $P$ and the vectors
$$g_1=\sqrt{2}e_7+e_1-f_1+e_5+f_5, \qquad g_2=\sqrt{2}f_7-e_3+f_3-e_5-f_5. $$
A computer calculation shows that the stabilizer of $R$ in $\fg_2\times \fg_2$ is one-dimensional. So the generic orbit has codimension at most one, as claimed.  \qed

\medskip\noindent {\it Remark.} As for the previous case we have not been able to establish whether the generic orbit has codimension $0$ or $1$. And again, extensive "random" computations suggest that there is no open orbit.  

\medskip
We conclude this section with one last identification, which can be seen as a variant of 
\cite[Theorem 3.1]{kuznetsov}.

\begin{prop}
$(OG(n,2n+1), \wedge^2\cU^\vee)\simeq (\PP^1)^n$. 
\end{prop}

\proof This is the variety of $n$-planes that are isotropic with respect to both a non-degenerate quadratic form $Q$ and a general skew-symmetric form $\omega$; we denote this variety by $X$. Since the dimension is odd, $\omega$ has a one-dimensional kernel $\langle e_0\rangle$, which is not isotropic with respect to $Q$. On its orthogonal $V:=e_0^\perp$, the restrictions of both $Q$ and $\omega$ are non-degenerate, and define isomorphisms $\theta_Q,  : V\lra V^\vee$ and $\theta_\omega,  : V^\vee\lra V$. For $\omega$ general, 
$\theta_\omega\circ \theta_Q$ will be regular semisimple. In particular, it stabilizes some plane $P_1\subset V$, which means that the orthogonal to $P_1$ in $V$ is the same with respect to either $Q$ or $\omega$. By induction, we can find a decomposition $V=P_1\oplus\cdots\oplus P_n$ which is orthogonal 
with respect to both $Q$ and $\omega$. Choosing a suitable basis $(e_i,f_i)$ for each plane $P_i$, we put the matrices of $Q$ and $\omega$ in the following simultaneous normal forms:
$$Q=\begin{pmatrix}
J & 0  & \cdots &0&0\\
0& J& \cdots  & 0&0\\
\cdots &&&&\cdots \\
0 & 0  & \cdots &J &0\\
0 & 0  & \cdots& 0 &1
\end{pmatrix}, \qquad 
\omega =\begin{pmatrix}
\theta_1K & 0  & \cdots &0&0\\
0& \theta_2K& \cdots & 0&0\\
\cdots &&&&\cdots \\
0 & 0 &\cdots &\theta_n K &0\\
0 & 0  & \cdots& 0 &0
\end{pmatrix},$$
for some nonzero scalars $\theta_1, \cdots ,\theta_n$, where 
$$J=\begin{pmatrix} 0&1\\ 1&0\end{pmatrix}, \qquad 
K=\begin{pmatrix} 0&1\\ -1&0\end{pmatrix}.$$
It is then convenient to embed $\CC^{2n+1}$ as a hyperplane in $V\simeq \CC^{2n+2}$ by adding a vector $f_0$ such that $Q(e_0,f_0)=0$ and $\omega(e_0,f_0)=-1$. We get a basis of isotropic vectors in $\langle e_0, f_0\rangle$ by letting $e_{n+1}=\frac{1}{\sqrt{2}}(e_0+f_0)$ and $f_{n+1}=\frac{1}{\sqrt{2}}(e_0-f_0)$. We can then identify $OG(n,2n+1)$ with one component of $OG(n+1,2n+2)$, by sending  
 $P_0=\langle e_1,\ldots ,e_n\rangle$ to $Q_0=\langle e_1,\ldots ,e_n,e_{n+1}\rangle$. Note that $P_0$ belongs to our locus $X$, that we want to describe by focusing first to a neighborhood of $Q_0$ in  $OG(n+1,2n+2)$; indeed, any maximal isotropic subspace $R \subseteq V$ transverse to  $Q_1=\langle f_1,\ldots ,f_n,f_{n+1}\rangle$ has a unique basis of the form 
$$q_i=e_i+\sum_{j=1}^n \alpha_{ij} f_j+\beta_i f_{n+1}, \quad 1\le i\le n,$$
$$q_{n+1}=e_{n+1}-\sum_{i=1}^n e_i,$$
with $\alpha_{ij}=-\alpha_{ji}$. We need to impose that $P=R\cap\CC^{2n+1}$ is also isotropic with respect to $\omega$.
A basis of $P$ is given by the vectors $p_i=q_i+\beta_i f_{n+1}$, 
$1\le i\le n$. We have 
$$\omega(p_i,p_j)=\theta_i(\alpha_{ij}+\gamma_i\gamma_j)-\theta_j(\alpha_{ji}+\gamma_j\gamma_i),$$
so that the isotropy condition completely determines $\alpha$:
$$\alpha_{ij}=\frac{\theta_i-\theta_j}{\theta_i+\theta_j}\gamma_i\gamma_j\qquad \forall i,j.$$
If we let $\gamma_{n+1}=1$, we can conclude that the skew-symmetric matrix $\pi$ of size $n+1$ that defines $R$ is of the form $\mu_{ij}\gamma_i\gamma_j$ for a fixed matrix $\mu$, depending only on $\theta$. 

Now, we need to remember (see \cite[section 2.3]{spinor-secant}) that the spin embedding of $OG(n,2n+1)= OG(n+1,2n+2)_+$ inside $\PP \Delta\simeq \PP \Delta_+$ can be described by sending $P$ to the point in projective space whose homogeneous coordinates are given by the Pfaffians $Pf_I(\pi)$ of all the skew-symmetric submatrices of $\pi$, defined by the indices $I=(i_1<\cdots <i_{2k})$ of their lines and columns. But because of the form of the matrix $\pi$, we clearly have
$$Pf_I(\pi)=Pf_I(\mu)\prod_{i\in I}\gamma_I.$$
The index set $I$ must be of even size, but since $\gamma_{n+1}=1$, we get in fact all the products of distinct entries of $\gamma$. 
So we exactly recover the classical parametrization of the image of the Segre embedding of  $(\PP^1)^n$ inside $\PP^{2^n-1}$.

\medskip
There remains to check that this is the full locus we are looking for. It is enough for this to check that the degree of $(\PP^1)^n$, that is $n!$, is equal to the degree of our locus. The latter can be computed as 
$$
d=\int_{OG(n,2n+1)}c_{top}(\wedge^2U^\vee)h^n=\frac{1}{2^n}
\int_{G(n,2n+1}c_{top}(\wedge^2U^\vee)c_{top}(S^2U^\vee)\sigma_1^n,$$
since the Pl\"ucker class restricted to $OG(n,2n+1)$ is twice the class that defines the spinor embedding. The top Chern classes are known to be 
$$c_{top}(\wedge^2U^\vee)=\sigma_{\delta(n-1)}, \qquad c_{top}(S^2U^\vee)=2^n\sigma_{\delta(n)},$$
where $\delta(n)=(n,n-1,\ldots ,2,1)$. Recall that the Schubert classes $\sigma_\lambda$, for $\lambda$ a partition
inscribed into the rectangle of size $n\times (n+1)$, provide a basis of the Chow ring of $G(n,2n+1)$. Moreover, this basis is Poincar\'e self-dual, up to the permutation sending $\lambda$ to the partition which is its complement in the rectangle. In particular, the class 
$\sigma_{\delta(n)}$ is its own Poincar\'e dual. We conclude that the degree $d$ is equal to the coefficient of $\sigma_{\delta(n)}$ inside the product of $\sigma_{\delta(n-1)}$ by $\sigma_1^n$. By the Pieri rule, the coefficient of $\sigma_\mu$ inside the product of $\sigma_{\lambda}$ by $\sigma_1^k$ is, in general, the number of ways one can number by $1,\ldots , k$ the complement of the partition $\lambda$ into the partition $\mu$, in such a way that the numbers increase on each line and each column. But the latter condition is clearly void for $\lambda=\delta(n-1)$ and $\mu=\delta(n)$, so $d=n!$ and the proof is complete. \qed 

\section{A non-trivial family}

According to \cite[Proposition 41]{sk}, the action of $Spin_{13}$ on the projectivized spin representation $\Delta$ is \emph{not} prehomogeneous: the generic stabilizer is $SL_3\times SL_3$ and the generic orbit has codimension one. 

Nevertheless, we observe two different interesting phenomena: in the next Proposition, we get a rigidity 
result in the adjoint variety of $Spin_{13}$; then we obtain a family of prime Fano manifolds of dimension at least one that compactify a finite quotient of $SL_3\times SL_3$
(probably $SL_3\times SL_3$ itself, but we have not checked this). That such a family of prime compactifications can exist looks quite remarkable.

\begin{prop}
$(OG(2,13), \cS)=Fl_3\cup Fl_3$
\end{prop}

\proof Note that $OG(2,13)$ embeds in $OG(2,14)$ as the zero locus of general section of the dual rank two tautological bundle $\cU^\vee$. Therefore, because of  Proposition \ref{doubleG2P2}, we just need to prove that the zero locus of a generic section of the restriction of $\cU^\vee$ to $G_2/P_2$ is a copy of $Fl_3$. But this follows from \cite[Lemma 4]{kr}.   \qed 

\begin{prop}
$(OG(3,13), \cS)$ is a non-trivial family of Fano manifolds of dimension $16$, index $5$ and Picard number one, which are equivariant compactifications of  $SL_3\times SL_3$, up to some finite group. 
\end{prop}

\noindent  {\it Betti numbers:} $(1,1,2,3,5,6,8,10,12,10,8,6,5,3,2,1,1)$.

\proof 
The Lie algebra stabilizer $\fs$ in $\fg_2\times \fg_2$ of a generic $P$ in $(OG(3,14), \cS_+)$ was computed in the proof of Proposition \ref{3,14}. For a generic hyperplane $H$ containing $P$, defining a copy of $Spin_{13}$ inside $Spin_{14}$, this $P$ will be a generic point of the manifold $X$ of type $(OG(3,13), \cS)$ defined by the same spinor, and its Lie algebra stabilizer in $\fso_{13}$ will just be the intersection in $\fso_{14}$ of $\fs$ and the stabilizer of $H$. For our explicit $P$, we choose $H=\langle e_1+e_4+e_7-f_7\rangle^\perp$, and check by an  explicit computation that this intersection is empty. This proves that the generic stabilizer of the action of $SL_3\times SL_3$ on $X$ is finite. 

On the other hand, another computer calculation shows that the Euler characteristic $\chi(X,TX)=15$. Since $H^0(X,TX)$ certainly contains $\fsl_3\times\fsl_3$, whose dimension is $16$, we deduce that $h^1(X,TX)\ge 1$, hence that our family is non-trivial.\qed 
 
\medskip\noindent {\it Remark.} By the same argument as for $(OG(3,14),\cS_+)$, which applies verbatim since a maximal torus of $SL_3$ is also a maximal torus of $G_2$, $(OG(3,13),\cS)$ admits a torus action with finitely many fixed points. Hence its cohomology is pure and it is a compactification of affine space.  

\medskip
In a similar way, we get the following statement:

\begin{prop}
$(OG(4,13), \cS)$ is a non-trivial family of Fano manifolds of dimension $22$, index $6$ and Picard number one.
\end{prop}

\noindent  {\it Betti numbers:} $(1,1,2,3,5,6,8,-,\ldots,-,8,6,5,3,2,1,1)$.

\medskip
Here again $\chi(X,TX)=15$, so we expect the family to be one dimensional, the automorphism groups being again $SL_3\times SL_3$, up to some finite group. 
But we cannot hope for an open orbit of the automorphism group, since the dimension of $X$ is bigger than $16$. 

\medskip\noindent {\it Remark.} More generally, one could call spinorial Fano manifolds the smooth zero loci of sections of spinor bundles on orthogonal Grassmannians $OG(k,2n)$ or $OG(k,2n+1)$, for any $k$ and $n$. Such zero loci must be non-empty for $n$ large and $k$ close enough to $n$, since the rank of spinor bundles grows like $2^{n-k}$. For $n>14$, one will get large families of prime Fano manifolds with large index, but the generic ones will presumably have trivial automorphism group, contrary to those that were considered in the present paper.

\end{document}